\documentclass{article}
\usepackage[utf8]{inputenc}
\usepackage{authblk}

\usepackage{tikz}
\usepackage{tikz-cd}
\usetikzlibrary{arrows}
\tikzcdset{arrow style=tikz,
diagrams={>={Stealth[round,length=4pt,width=6pt,inset=3pt]}}}

\usepackage{amsthm}
\usepackage{amsmath} 
\usepackage{amssymb}
\usepackage{amsfonts}
\usepackage{faktor}
\usepackage[T1]{fontenc}
\usepackage{faktor}
\usepackage{xfrac}
\usepackage{dsfont}
\usepackage{color}
 
\usepackage[english]{babel}
 \usepackage{hyperref}
 
\newtheorem{theorem}{Theorem}
\newtheorem{definition}[theorem]{Definition}

\newtheorem{lemma}[theorem]{Lemma}
\newtheorem{notation}[theorem]{Notation}

\newtheorem{proposition}[theorem]{Proposition}

\newtheorem{example}[theorem]{Example}
\newtheorem{remark}[theorem]{Remark}
\newtheorem{corollary}[theorem]{Corollary}

\newtheorem*{theorem*}{Theorem}

\DeclareMathOperator{\Sing}{Sing}
\DeclareMathOperator{\Crit}{Crit}
\DeclareMathOperator{\Fitt}{Fitt}
\DeclareMathOperator{\Hess}{Hess}

\title{Invariants of Non-Isolated Singularities of Hypersurfaces}
\author{Yotam Svoray}

\date{    }

\AtEndDocument{\bigskip{\footnotesize%
  \textsc{ Department of Mathematics, Ben Gurion University of the Negev,
Be’er Sheva, Israel} \par  
  \textit{E-mail address}: \texttt{svoray@math.utah.edu} 
}}

\begin{document}

\maketitle

\begin{abstract}
        In this paper we generalize some results by Siersma, Pellikaan, and de Jong regarding morsifications of singular hypersurfaces whose singular locus is a smooth curve, and present some applications to the study of Yomdin-type isolated singularities. In order to prove these results, we discuss the transversal discriminant of such singularities and how it relates to other algebraic and topological invariants.
\end{abstract}

  \tableofcontents

\section{Introduction}

Understanding the singular points of a singular hypersurface allows us to better understand the geometry and the topology of the hypersurface. An important tool used to understand such singular points is to deform the hypersurface and see how such singular points behave under "special" deformations. \\  

The goal of this paper is to study singular hypersurfaces $(V(f),0) \subset (\mathbb{C}^n,0)$ such that $\Sing(V(f))=V(I)$ is a smooth curve germ, $f \in I^p \setminus I^{p+1}$ for some $p\geq 2$, and its generic transversal type is an ordinary multiple point. Specifically, how such singular hypersurfaces behave under a special kind of deformation, and what bounds we can conclude on some topological and algebraic invariant of $V(f)$.  We summarize the results we prove in the following theorem, which is a combination of Theorem~\ref{prop:morsif}, Theorem~\ref{thm:main_Jac_gen}, Theorem~\ref{thm:delta_trans_bound}, and Corollary~\ref{cor:mu_rel_mors}. \\

\begin{theorem*}
    Let $f \in I^p \setminus I^{p+1}$ be a germ of analytic function (where $p\geq 2$) with $\Sing(V(f))=V(I)$ such that its generic transversal type is an ordinary multiple point.  Then $f$ has a relative morsification $f_t$ (see Definition~\ref{def:morsification}) and for every large enough $k$ and small enough $t_0$ we have that:
    \begin{enumerate}
        \item If $n \neq 4$ then 
        \begin{equation*}
    j(f) \geq \# A_1(f_{t_0}) +\deg(\Delta^\perp(f)),
        \end{equation*}
          where $\# A_1(f_{t_0})$ denoted the number of Morse points $f_{t_0}$ has outside $V(I)$ (see Notation~\ref{notation:white_mountain}), $j(f)$ is defined in Definition~\ref{def:for_a _moment}, and $\deg(\Delta^\perp(f))$ is the degree of the transversal discriminant of $f$ as a Cartier divisor (see Definition~\ref{def:trans_disc}). 
    \item $\delta(f) \geq \deg(\Delta^\perp(f))$, where $\delta(f)$ is defined in Section~\ref{sec:milnor}.  
    
    \item The Milnor number of $f+x_n^k$ is finite and bounded below by 
    \begin{equation*}
            \#A_1(f_{{t_0}}) + (k-1)(p-1)^{n-1} + 2\deg(\Delta^\perp(f)). 
    \end{equation*}

    \end{enumerate}
\end{theorem*}

We now describe some historical background and some motivation for this paper. The study of deformations of singular hypersurfaces started with the study of singular hypersurface whose singular locus is a single point (which we may assume is at the origin). Such singular hypersurfaces are said to have isolated singularities.
Much is known about isolated singularities and about their topology and geometry, as discussed and reviewed in~\cite{arnol1993singularity,de2013local,greuel2007introduction}.\\

An important result regarding isolated singularities is the relationship between the Milnor number of an analytic germ $f \in \mathbb{C}\{ x_1, \dots, x_n \}$ (which equals to $\mu(f)=\dim_\mathbb{C} (\sfrac{\mathbb{C} \{x_1, \dots, x_n\}}{Jac(f)})$) and the number of Morse points in a morsification of $f$, which is a special kind of deformation of $f$ (for more details see Section 3.8 in~\cite{ebeling2007functions}). Therefore, a natural question to ask is how to generalize this relationship to hypersurfaces $V(f)$ whose singular locus is not isolated. \\

Siersma~\cite{siersma1983isolated} studied the Milnor fiber of $f$ where $\Sing(V(f))=V(I)$ such that the germ of $f$ at $(0, \dots, 0,x_n)$ is equivalent to an $A_\infty$ singularity for every small $x_n \neq 0$.  Siersma proved that the Milnor fiber of $f$ is homotopy equivalent to a bouquet of spheres $S^{n-1}$. In order to compute the number of spheres in this bouquet, Siersma studied deformations of such hypersurfaces into hypersurfaces with only $A_1$, $A_\infty$, and $D_\infty$ singularities. Siersma conjectured that the number of $A_1$ points plus the number of $D_\infty$ points in such a deformation,  that is, $\# A_1(f) + \#D_\infty(f)$, does not depend on the deformation and can be expressed as an algebraic invariant of $f$.   \\

Pellikaan~\cite{pellikaanhypersurface, pellikaan1990deformations}, generalized Siersma's ideas regarding isolated line singularities to the case where the singular locus of $V(f)$ is a reduced curve defined by some ideal $I$. Using techniques from commutative and homological algebra developed in~\cite{pellikaan1988projective}, together with analytic and topological tools, Pellikaan was able to prove Siersma's conjecture: $ \# A_1(f) +\# D_\infty(f) = \dim_\mathbb{C} (\sfrac{I}{Jac(f)})$. In addition, Pellikaan~\cite{pellikaan1989series} was able apply this result in order to compute the Milnor number of a related isolated singularities, inspired by the work of Yomdin~\cite{iomdin1974complex} and later L\^{e}~\cite{le1980}.  Pellikaan's work on Siersma's conjecture was later expanded upon by de Jong~\cite{de1988some} to the case where the transversal type of the singular locus is determined by some fixed simple ($ADE$) isolated singularity. \\

Therefore, inspired by Siersma, de Jong, and Pellikaan, we wish to generalize the formula $ \# A_1(f) +\# D_\infty(f) = \dim_\mathbb{C} (\sfrac{I}{Jac(f)})$ (which we call the Siersma-Pellikaan-de Jong formula) and other related results to the non reduced case, i.e., where the singular locus is a non-reduced scheme. Yet, an obstacle that we face in this case is that the transversal type of $V(f)$ is not a simple singularity, but rather a multiple point. So, in order to understand the transversal type of our singular locus, we need to look at a special Cartier Divisor of $\Sing(V(f))$ called the tranversal discriminant of $V(f)$, denoted $\Delta^\perp(f)$, as described by Kerner and N\'{e}methi in~\cite{kerner2017discriminant, kazarian2013discriminant}, which we discuss in Section~\ref{section:trans_disc}. \\ 

In this paper, by studying a special kind of deformations of $f$, which we call a relative morsification, as constructed in Section~\ref{section:rel_mors}, we are able to obtain a generalization of the the Siersma-Pellikaan-de Jong formula, as presented in the theorem above. In addition, as in Pellikann~\cite{pellikaan1989series}, we show how we can use the tools used to prove this generalizatiom in order to compute and bound the Milnor number of isolated singularities of the form $f+x_n^k$. In Section~\ref{section:closing} we see why our results are indeed generalizations and discuss what are the challenges in the non reduced case.\\

Throughout the text, unless stated otherwise, we assume that $f \in I^p \setminus I^{p+1}$ for some $p\geq 2$ with $\Sing(V(f))=V(I)$ such that its generic transversal type is an ordinary multiple point. Note that since $V(I)$ is a smooth curve germ, we can assume that $I=\langle x_1, \dots, x_{n-1}\rangle$. We denote by $\mathbb{C}\{\underline{x}\} = \mathbb{C}\{x_1, \dots,x_n\}$ the ring of germs of analytic functions in $x_1, \dots, x_n$ over $\mathbb{C}$. In addition, for an ideal $\mathfrak{a}\subset \mathbb{C}\{\underline{x}\}$, we view $V(\mathfrak{a})$ both as a subset of $\text{Spec}(\mathbb{C}\{\underline{x}\})$ and as a germ in $(\mathbb{C}^n,\underline{0})$, as in the R\"uckert Nullstellensatz (See Theorem 1.72 in~\cite{greuel2007introduction} for more details).  \\

\textbf{Acknowledgments.} The results presented in Section~\ref{section:rel_mors}, in addition to some of the general ideas in this paper, were obtained under the guidance of of Dmitry Kerner, as part of the author's M.Sc. thesis, and we thank him for his guidance, support, and contribution. We wish to thank Ishai Dan-Cohen, Srikanth Iyengar, Karl Schwede, Kurt Vinhage, and Ariel Yadin for enjoyable and productive discussions. In addition, we thank Matthew Bertucci, Edmund Karasiewicz, and Karl Schwede for their input and comments on
the different drafts of this paper. The Author was partially supported by the ISF grant 621/21 of Ishai Dan-Cohen and by the ISF grant 1910/18 of Dmitry Kerner.\\

\section{Transversal Discriminant}\label{section:trans_disc}

In this section we summarize and review some basic properties of the transversal discriminant of a hypersurface, based upon~\cite{kerner2017discriminant, kazarian2013discriminant}. We start with a motivational example, the discriminant of a polynomial in one variable (note that we in fact use this example in the proof of Theorem~\ref{prop:morsif}).

\begin{definition}\label{def:discriminant_poly}
Let $p(z)=\sum_{j=0}^n a_j z^j \in \mathbb{C}[z]$ be a complex polynomial (of degree $n \geq 2$) whose roots are $\alpha_1, \dots, \alpha_n \in \mathbb{C}$. Then the \textbf{discriminant of $p$} is defined to be 
\begin{equation*}
    \Delta_z(p)=a_n^{2n-2} \prod_{i < j} (\alpha_i -\alpha_j)^2
\end{equation*}
\end{definition}

\begin{proposition}\label{property:discriminant_resultant}\label{rem:disc_red_one_var}
Let $p(z)=\sum_{j=0}^n a_j z^j \in \mathbb{C}[z]$ be a polynomial of degree $n$. Then $\Delta_z(p)=0$ if and only if $p$ is non reduced, and we have that 
\begin{equation*}
    \Delta_z(p)=\frac{(-1)^{\frac{n(n-1)}{2}}}{a_n} Res_z(p,\frac{dp}{dz}),
\end{equation*}

where $Res_z(p,\frac{dp}{dz})$ is the resultant of the polynomials $p$ and $\frac{dp}{dz}$.

\end{proposition}

\begin{proof}
See Section 1.B of Chapter 12 of~\cite{gelfand2008discriminants}. 
\end{proof}

\begin{remark}\label{remark:discriminant_alg}
\begin{enumerate}
    \item \textup{If we look at every polynomial of degree $n$ as its tuple of coefficients, then by Proposition~\ref{property:discriminant_resultant}, the discriminant defines a polynomial map $ \mathbb{C}^{n+1} \to  \mathbb{C}^1$}
    \item\textup{One can generalize the concept of the discriminant of a polynomial to polynomials in a number of variables. For proofs and more details regarding the discriminant and the resultant, see~\cite{gelfand2008discriminants}. 
}
\end{enumerate}
\end{remark}

One might ask if we can capture properties other than reducability in a similar way. Specifically, the property we are interested in is how the transversality type of our hypersurface along its singular locus changes and degenerates:\\

Let $(X,0)=(V(f),0) \subset (\mathbb{C}^n,0)$ be a hypersurface with $Z=\Sing(V(f))$, where we assume that $Z$ is smooth. For every $o \in Z$ we consider a smooth germ $(L,o) \subset (\mathbb{C}^n,o)$ such that $(L,o) \cap (Z,o) =\{o\}$ and $(L,o)$ is transversal to $(Z,o)$. By looking at $(X \cap L,o)$ and changing $o$, we can see that its type is generically constant in some sense, but there are still points in which this transversal type degenerates.\\

\begin{example}
\textup{Let $X=V(f) \subset \mathbb{C}^3$ where $f(x,y,z)=x^p-y^pz$. Note that $\Sing(V(f))=V(x,y)$. For every $z_0 \in \mathbb{C}$, we can view $f_{z_0}(x,y)=f(x,y,z_0)$ as a polynomial in two variables which defines a plane curve singularity. For every $z_0 \neq 0$, we have that $f_{z_0}$ defines an ordinary multiple point with $p$ tangents, but for $z_0=0$ we have that $f_0(x,y)=x^p$ which is a unique line of multiplicity $p$. So $z_0=0$ is a point in $V(x,y)$ where the transversal type degenerates.}
\end{example}

This intuition leads us to the definition of the transversal discriminant of $f$.

\begin{definition}[2.2.5 in~\cite{kazarian2013discriminant}]\label{def:trans_disc}
 Given $f \in \mathbb{C}\{\underline{x}\}$ with $\text{Sing}(V(f))=V(I)$, the \textbf{transversal discriminant} of $f$, is defined to be
\begin{align*}
    \Delta^\perp(f) = V(\Fitt_0 (\pi_* \mathcal{O}_{\Crit(\pi)} )),
\end{align*}
where $\pi \colon \widetilde{V(f)} \to V(f)$ is the restriction of the blow up map $Bl_{V(I)} ( \mathbb{C}^n) \to \mathbb{C}^n$ to the strict transform of $V(f)$, and $\Fitt_0$ is the zeroth fitting ideal.
\end{definition}

\begin{example}\label{example:transversal}
\textup{Let $f(\underline{x})=x_nx_{n-1}^p+\sum_{i=1}^{n-1} x_i^p \in I^p$ for some $p \geq 2$, where $I=\langle x_1, \dots, x_{n-1} \rangle$. Then $\Sing(V(f))=V(I)$, and thus  
\begin{equation*}
     Bl_{V(I)}(\mathbb{C}^n)=\{(x_1,\dots,x_n,[\sigma_1 \colon \dots \colon \sigma_{n-1}]) \colon \forall i, j,  x_i \sigma_j=x_j \sigma_i\}. 
\end{equation*}
So, if we denote by $E$ the exceptional divisor of $Bl_{V(I)}(\mathbb{C}^n)$, then}
\begin{equation*}
    E=\{(0,\dots ,0,x_n,[\sigma_1 \colon \cdots \colon  \sigma_{n-1}])\},
\end{equation*}

\textup{and we can compute that}

\begin{equation*}
    \widetilde{V(f)} \cap E= \{x_n\sigma_{n-1}^p+\sum_{i=1}^{n-2}\sigma_i^p=0, x_1=\cdots =x_{n-1}=0\}.
\end{equation*}

\textup{Now, by restricting the blow up map to $\pi \colon \widetilde{V(f)} \cap E \to V(f)$ and looking at $\pi$ with respect to each chart of the blow up, we can compute that $\Crit(\pi)= V(x_n, \sigma_1 ^{p-1}, \dots, \sigma_{n-2}^{p-1} )$ and so $     \mathcal{O}_{\Crit(\pi)} =\sfrac{\mathbb{C}\{x_n,\sigma_1, \dots, \sigma_{n-2} \}}{\langle x_n, \sigma_1^{p-1}, \dots,  \sigma_{n-2}^{p-1} \rangle}.$ Yet we have an isomorphism of $\mathbb{C}[x_n]-$modules
$\sfrac{\mathbb{C}[x_n,\sigma_1, \dots, \sigma_{n-2} ]}{\langle x_n, \sigma_1^{p-1}, \dots,  \sigma_{n-2}^{p-1} \rangle} \cong 
     \bigoplus_{0 \leq i_1, \dots, i_{n-2} \leq p-1} \mathbb{C} \cdot \sigma_1^{i_1} \cdots \sigma_{n-2}^{i_{n-2}} \cong 
     (\sfrac{\mathbb{C}\{x_n\}}{\langle x_n \rangle})^{\oplus (p-1)^{n-2}}$, and so we have that $\pi_* \mathcal{O}_{\Crit(\pi)}=(\sfrac{\mathbb{C}\{x_n\}}{\langle x_n \rangle})^{\oplus (p-1)^{n-2}}$. Therefore, we can conclude that  }

\begin{equation*}
    \Fitt_0(\pi_*\mathcal{O}_{\Crit(\pi)})=\Fitt_0(\sfrac{\mathbb{C}[x_n]}{\langle x_n \rangle})^{(p-1)^{n-2}} =\langle x_n^{(p-1)^{n-2}} \rangle. 
\end{equation*}
\end{example}

We end this section with a few remarks which play a crucial role in the proofs later on. 

\begin{remark}\label{remark:trans_poly}
\begin{enumerate}
    \item \textup{Proposition 4.4 in~\cite{kerner2017discriminant} tells us that $\Delta^\perp(f) \subset \Sing(V(f))$ is a Cartier Divisor. In fact, an explicit generator for $I(\Delta^\perp(f))$ (in the case we are interested in) is given in section 4.1.1 in~\cite{kerner2017discriminant}, and in~\cite{kazarian2013discriminant}, the equivalence class of $\Delta^\perp(f)$ as a Cartier divisor in the Picard group $Pic(\text{Sing}(V(f)))$ has been studied and computed using the Thom-Porteus Formula.}
    \item \textup{A useful conclusion (from item 1 of this remark) is that given a polynomial $f \in \mathbb{C}[\underline{x}]=\mathbb{C}[x_1, \dots, x_n]$, then this explicit generator of $\Delta^\perp(f)$ is a polynomial in the coefficients of $f$, and if $d$ is the total degree of $f$, then the degree of this polynomial is bounded by $\tilde{n}=(n-1)(d-1)^{n-2}$ (see Section 5.1 in~\cite{kerner2017discriminant}). Therefore, this induces an algebraic map $ \mathbb{C}^N \to  \mathbb{C}^{\tilde{n}}$, where we associate (by looking at the coefficient) $\mathbb{C}[\underline{x}]_{\leq d} =\{ f \in \mathbb{C}[\underline{x}] \colon \deg_{\text{total}}(f) \leq d\}$ with $ \mathbb{C}^N$ (for $N = \binom{d+n}{n}$) and $\mathbb{C}[x]_{\leq \tilde{n}} = \{ f \in \mathbb{C}[x] \colon \deg(f) \leq \tilde{n}\}$ with $ \mathbb{C}^{\tilde{n}}$.  }
\end{enumerate}
\end{remark}

\begin{remark}\label{rem:top_trans_disc}
\textup{Definition~\ref{def:trans_disc} gives us an additional topological understanding of the transversal discriminant (assuming that the generic transversal type of $V(f)$ along its singular locus is an ordinary multiple point): $\Delta^\perp(f)$ is empty if and only if the transversal sections of $V(f)$ along its singular locus are topologically equisingular, i.e., for every $p_1$ and $p_2$ in $\text{Sing}(V(f))$ (which we assume is smooth), if we consider $(L_1,p_1) \subset (\mathbb{C}^n,p_1)$ and $ (L_2,p_2) \subset (\mathbb{C}^n,p_2)$ such that $(L_i,p_i) \cap (\text{Sing}(V(f)),p_i) =\{p_i\}$ and $(L_i,p_i)$ is transversal to $(\text{Sing}(V(f)),p_i)$, for $i=1,2$, then the germs $(V(f) \cap L_1, p_1)$ and $(V(f) \cap L_2, p_2)$ are homeomorphic. For a deeper discussion on this result, see Section 2.2.3 in~\cite{kazarian2013discriminant} and Section 2.7 in~\cite{kerner2017discriminant}. }
\end{remark}


\begin{remark}\label{prop:disc_flat_defor}
    \textup{An important property of the transversal discriminant is that the transversal discriminant behaves nicely under flat deformations. Proposition 5.1 in~\cite{kerner2017discriminant} tells us that given $f \in \mathbb{C}\{\underline{x}\}$ and $f_t \in \mathbb{C}\{\underline{x},t\}$ a flat deformation of $f$ which deforms the singular locus of $V(f)$ flatly and the generic multiplicity of $f$ along its singular locus remains constant, then the family $\{\Delta^\perp(f_t)\}_t$ is flat and $\deg(\Delta^\perp(f_{t_0})) = \deg(\Delta^\perp(f))$ for every $t_0$, where the degree of $\Delta^\perp(f)$ is the degree of the transversal discriminant of $f$ as a Cartier divisor. For example, let $f(x,y,z)=x^p+y^pz$ and let $f_t(x,y,z)=x^p+y^pz+txy^{p-1}$. Then, one can compute that for every small $t$ we have that }
    \begin{equation*}
    \Delta^\perp(f_t)=\{z^{p-1}-\Big(\frac{t}{p}\Big)^p (1-p)^{p-1}=0\}
\end{equation*}
\textup{while $\Delta^\perp(f)=\{z^{p-1}=0\}$.}

\end{remark}

\section{Relative Morsification}\label{section:rel_mors}

In this section we discuss how we can deform $f$ in a specific way (as we define in Definition~\ref{def:morsification}), which plays a crucial role in the proof of Theorem~\ref{thm:main}, and later Theorem~\ref{thm:main_Jac_gen} and Theorem~\ref{thm:delta_trans_bound}. The concept of a relative morsification generalizes the concept of a morsification of a hypersurface $V(f)$ with a reduced line singularity, as presented in~\cite{siersma1983isolated, pellikaan1990deformations, de1988some}, as we discuss in Section~\ref{section:closing}. 


\begin{definition}\label{def:morsification}
\begin{enumerate}
    \item Given a point $\underline{y} \in\Crit(f)$, we say that $\underline{y}$ is a \textbf{Morse critical point of $f$} (or an $A_1$ point of $f$) if $\Hess(f)(\underline{y})$, the Hessian matrix of $f$ at the point $\underline{y}$, is invertible. 
    \item  We say that $f \in \mathbb{C}\{\underline{x}\}$ is  \textbf{Morse outside of $V(I)$} if $\text{Sing}(V(f))=V(I)$ and $f$ has only Morse critical points in a small neighborhood outside its singular locus.
    \item Let $f \in I^p \setminus I^{p+1}$ such that $\Sing(V(f))=V(I)$. We say that a flat deformation $f_{t} \in \mathbb{C}\{\underline{x},t\}$ of $f$ is a \textbf{relative morsification} of $f$ if $f_0=f$ and there exists some neighborhood $0 \in U \subset \mathbb{C}$ such for every $0 \neq t_0 \in U$:
\begin{enumerate}
    \item $\Sing(V(f_{t_0}))=V(I)$ and $f_{t_0} \in I^p \setminus I^{p+1}$.
    
    \item  ${f}_{t_0} \in I^p$ is Morse outside of $V(I)$. 
    \item The transversal discriminant $\Delta^\perp ({f}_{t_0})$ is a reduced subscheme of $V(I)$.
\end{enumerate}
\end{enumerate}

\end{definition}

\begin{remark}
    \textup{An analogue concept of a relative morsification has been studied from a topological point of view in Bobadillia~\cite{de2004relative} which also generalizes the concept of a morsification presented in~\cite{siersma1983isolated, pellikaan1990deformations, de1988some}. }
\end{remark}

\begin{example}\label{ex:def_mors_stuff} \label{example:relative} \label{example:table_mors}
\begin{enumerate}
    \item \textup{The following table contains a few examples of relative morsifications (reducability of the transversal discriminant follows from similar computations to the ones preformed in Example~\ref{example:transversal}): }
    \begin{center}
    \hspace*{-2.3cm}
    \begin{tabular}{ | l | l |}
    \hline
    $f$ & \textup{A relative morsification of} $f$ \\ \hline \hline
     $x^p+y^pz$ & $t(x^2y^{p-2}+y^2x^{p-2})+y^pz+x^p$ \\ \hline
     $x^p+y^pz^q+y^{p+1}$ & $t(x^2y^{p-2}+y^2x^{p-2})+y^p(z^q-tz)+y^{p+1}+x^p$  \\ \hline
     $x^pz^{q_1}+y^pz^{q_2}+y^{p+1}+x^{p+1}$ & $x^p(z^{q_1}-tz)+y^p(z^{q_2}-tz)+y^{p+1}+x^{p+1}+t(x^2y^{p-2}+y^2x^{p-2})$ \\ \hline 
     $\prod_{i=1}^n (x^{p_i}+y^{p_i}z)$ & $\prod_{i=1} ^n (t(x^2y^{p_i-2}+y^2x^{p_i-2})+y^{p_i}z+x^{p_i})$ \\ \hline
      $\prod_{i=1}^n(x^{p_i}+y^{p_i}z^{q_i}+y^{p_i+1})$ & $\prod_{i=1}^n (t(x^2y^{p_i-2}+y^2x^{p_i-2})+y^{p_i}(z^{q_i}-tz)+y^{p_i+1}+x^{p_i})$ \\ \hline 
     $\sum_{i=1}^{n-2}  x_i^p +x_{n-1}^px_n^q +x_{n-1}^{p+1}$ & $\sum_{i=1}^{n-2}  x_i^p +x_{n-1}^p(x_n^q-tx_n) +x_{n-1}^{p+1} + t(\sum_{i \neq j <n} x_i^2 x_j^{p-2})$ \\ \hline
     $\sum_{i=1}^{n-2} (x_i^p x_n^{q_i}+x_i^{p+1})$ & $\sum_{i=1}^{n-2} (x_i^p(x_n^{q_i}-tx_n)+x_i^{p+1}) + t(\sum_{i \neq j <n} x_i^2 x_j^{p-2})$ \\ \hline 
    \end{tabular}
    \label{tab:my_label}
    \end{center}
    \item \textup{Let $f(\underline{x})=\sum_{i=1}^n x_i^p$ where $p>1$. Then $f$ is Morse outside $V(I)$ as it has no non-singular critical points and $\Delta^\perp(f)=\emptyset$ since its transversal type does not change along its singular locus. }
    \item \textup{The deformation $f_t(x,y,z) = x^{p+1}+y^p +tx^p$ is not a relative morsification since for every $t_0 \neq 0$ we have that $(\frac{-tp}{p+1}, 0, z_0)$ is a non-Morse critical point for every $z_0$ which does not belong to $V(I)$.  }
\end{enumerate}
\end{example}

\begin{remark}\label{rem:empty_discr}
 \textup{If $\Delta^\perp(f)$ is empty then any morsification of $f$ has no Morse critical points outside $V(I)$. This is true since its blow up is smooth and as a deformation of a smooth hypersurface is smooth, and the Milnor number stays zero (see Section 2.6 in~\cite{greuel2007introduction} for more details). But the Milnor number of an isolates singularity is the number of Morse points in any morsification of $f$ (see Proposition 3.19. in~\cite{ebeling2007functions}). In fact, we can conclude that if $\Delta^\perp(f) =\emptyset$ then for every deformation $f_t$ (as in Remark~\ref{prop:disc_flat_defor}) and for every small ennough $t_0$ we have that $\Delta^\perp(f_{t_0}) =\emptyset$.}

\end{remark}



\begin{theorem}\label{prop:morsif}
For every $f\in I^p \setminus I^{p+1}$ with $\Sing(V(f))=V(I)$ such that its generic transversal type is an ordinary multiple point, there exists a relative morsification. Moreover, it can be taken in the form $f_t=f + tg$ for some $g \in I^p$.
\end{theorem}

\begin{proof}

We prove this proposition in two steps. In the first step we assume that $f$ is a polynomial in $\mathbb{C}[\underline{x}]$ and we find a polynomial relative morsification for $f$. In the second step we prove the result (for analytic functions), based upon the first step.   
\begin{itemize}
    \item \underline{Step 1: Assume that $f \in  \mathbb{C}[\underline{x}]$.}
\end{itemize}
Denote the (total) degree of $f$ with respect to $x_1, \dots, x_{n}$ by $d$. 
\\

By viewing every polynomial in $\mathbb{C}$ as its tuple of coefficients, we identify $\mathbb{C}[x_1, \dots, x_{n}]_{\leq d}$ (the set of all polynomials $g \in \mathbb{C}[\underline{x}]=\mathbb{C}[x_1, \dots, x_n]$ whose degree  with respect to $x_1, \dots, x_{n}$ does not surpass $d$), with the affine space $ \mathbb{C}^N$ where $N=\binom{d+n}{n}$. Now, viewed as subsets of $\mathbb{C}^N$,

\begin{enumerate}
    \item Denote by $\Sigma_1$ the set of polynomials $g \in I^p \cap  \mathbb{C}[x_1, \dots, x_{n}]_{\leq d}$  such that $g$ is not Morse outside $V(I)$, as in Definition~\ref{def:morsification}.
    \item Denote by $\Sigma_2$ the set of polynomials $g \in I^p \cap \mathbb{C}[x_1, \dots, x_{n}]_{\leq d}$  such that $\Delta^\perp(g)=\langle h \rangle$ where $h \in \mathbb{C}[x_n]$ is not reduced or $h=0$.
\end{enumerate}

We divide this step into three parts: In the first two parts (1.1 and 1.2) we show that $\overline{\Sigma_1}$ (the closure of $\Sigma_1$ in $\mathbb{C}^N$ with respect to the Euclidean topology) and $\Sigma_2$ are algebraic subsets which are proper subsets of $I^p$ (where we consider $I^p$ as a subset of $ \mathbb{C}^N$). In the third part (1.3) we construct the desired relative morsification of $f$, using the fact that $f_t$ is a relative morsification if and only if for every small enough $t_0$, $f_{t_0} \in I^p \setminus (\Sigma_1 \cup \Sigma_2)$ and $f_0=f$. 
\\

We are interested in $\overline{\Sigma_1}$ since $\Sigma_1$ is not closed in the classical topology. This is true because if $g_t$ has a unique non-Morse point at $(t,0,..,0)$, then as $t \to 0$ we have that $g_0 \notin\Sigma_1$.\\ 

Note that for every ideal $\mathfrak{a} \subset \mathbb{C}[\underline{x}]$ we have that  the (set theoretical) intersection $\mathfrak{a} \cap \mathbb{C}[\underline{x}]_{\leq d}$ is a complex vector space, and thus it is an algebraic subset of $ \mathbb{C}^N$ (as it is the intersection of linear hypersurfaces). 
\\

By Example~\ref{example:relative}, for every $t$, the function $\sum_{i=1}^{n-2}  x_i^p +x_{n-1}^p(x_n^q-tx_n) +x_{n-1}^{p+1} + t(\sum_{i \neq j <n} x_i^2 x_j^{p-2})$ has a reduced non trivial transversal discriminant and no non-Morse critical points outside of $V(I)$, which tells us that $\Sigma_1 \cup \Sigma_2 \neq I^p$.
\\


\vspace{4mm}

\underline{\textbf{Part 1.1}}:  Note that since $\Sigma_1$ is a cone in $ \mathbb{C}^N$ (that is, closed under multiplication by a complex scalar), $\overline{\Sigma_1}$ is a cone as well. Therefore it is enough to prove that $\mathbb{P}(\overline{\Sigma_1}) \subset \mathbb{P}^{N-1}(\mathbb{C})$ is a projective algebraic subset. Now, look at the set 
\begin{equation*}
    \Sigma_{1, pt} =\{([g],a) \in \mathbb{P}^{N-1}(\mathbb{C}) \times  \mathbb{C}^n \colon \text{$a \notin V(I)$ is not a Morse point of $g \in I^p$}\}. 
\end{equation*}

Note that $\Sigma_{1, pt}$ is well defined since $\Sigma_1$ is a cone. Take the compactification $\iota \colon  \mathbb{P}^{N-1}(\mathbb{C})  \times  \mathbb{C}^n \to \mathbb{P}^{N-1}(\mathbb{C})\times \mathbb{P}^n(\mathbb{C})$ defined by
\begin{equation*}
    \iota ([\sigma_1 \colon \dots \colon \sigma_N], z_1, \dots, z_n) = ([\sigma_1  \colon \dots \colon \sigma_N], [z_1 \colon \dots \colon z_n \colon 1]),
\end{equation*}

and let $\overline{\iota(\Sigma_{1,pt})}$ be the Zariski closure of the image of $\Sigma_{1,pt}$ under this embedding.  Note that
\begin{equation*}
    \Sigma_{1,pt} \subsetneq Y= \{ ([g],a) \colon \nabla(g)(a)=0 = \det(\Hess(g)(a)) \}\subsetneq \mathbb{P}^{N-1}(\mathbb{C}) \times  \mathbb{C}^n.
\end{equation*}

Thus we have that $\overline{\iota(\Sigma_{1,pt})} \neq  \mathbb{P}^{N-1}(\mathbb{C})\times \mathbb{P}^n(\mathbb{C})$, as it must be contained inside $\iota(Y) \subsetneq \mathbb{P}^{N-1}(\mathbb{C})\times \mathbb{P}^n(\mathbb{C})$.\\  

Denote by $\pi \colon \mathbb{P}^{N-1}(\mathbb{C})\times \mathbb{P}^n(\mathbb{C}) \to \mathbb{P}^{N-1}(\mathbb{C})$ the projection map.
Since $\overline{\iota(\Sigma_{1,pt})}$ is an algebraic subset, by Theorem 3.12. in~\cite{harris2013algebraic} we have that
$\pi(\overline{\iota(\Sigma_{1,pt})})$ is an algebraic subset as well. 
But since  $\pi(\iota(\Sigma_{1,pt}))= \mathbb{P}(\Sigma_1)$, we have that $\pi(\overline{\iota(\Sigma_{1,pt})})=\mathbb{P}(\overline{\Sigma_1})$, and we can conclude that $\mathbb{P}(\overline{\Sigma_1})$ is an algebraic set. 



\vspace{4mm}
\underline{\textbf{Part 1.2}}: By Proposition~\ref{remark:trans_poly} we know that $\Delta^\perp(g) \subset \Sing(V(g))$ is an algebraic subset for every $g \in I^p$. Denote by $h_g$ the defining generator of $I(\Delta^\perp(g))$, as mentioned in Remark~\ref{remark:trans_poly}, and note that $h_g \in \mathbb{C}[x_n]$. \\

Let $\Phi \colon  I^p \to  \mathbb{C}^1$ be the map $\Phi(f)=\Delta_{x_n}(h_f)$, where $\Delta_{x_n}$ is defined in Definition~\ref{def:discriminant_poly} and, as before, we view $I^p$ as a subset of $\mathbb{C}^N$. $\Phi$ is a polynomial map since by Remark~\ref{remark:discriminant_alg}, $\Delta_{x_n}$ defines a polynomial map, and, by Remark~\ref{remark:trans_poly}, the map $f \mapsto h_f$ is a polynomial map as well.\\

Thus, since we have that $\Phi^{-1}(\{0\})=\Sigma_2$ and since $\{0\} \subset  {\mathbb{C}}^1$ is an algebraic subset, then we get that $\Sigma_2 \subset I^p$ is an algebraic subset.

\vspace{4mm}

\underline{\textbf{Part 1.3}}: By the previous parts, $\overline{\Sigma_1} \subsetneqq I^p \cap \mathbb{C}[\underline{x}]_{\leq d}$ and $\Sigma_2 \subsetneqq I^p \cap \mathbb{C}[\underline{x}]_{\leq d}$ are algebraic subsets. Therefore there exists some $g \in I^p$ such that $g \notin \Sigma_1 \cup \Sigma_2$, because $I^p$ is a vector space which is irreducible as an algebraic subset. Denote $f_t=f+tg$. Then $f_t$ is a relative morsification of $f$ since $\{f_t \colon t \in \mathbb{C}\} \subset  {\mathbb{C}}^N$ is a line and thus must intersect $\overline{\Sigma_1} \cup \Sigma_2$ in only finitely many points. 
 \begin{itemize}
     \item \underline{Step 2 ($f\in \mathbb{C}\{\underline{x}\}$):} 
 \end{itemize} 
First, note that since $V(I)$ is a smooth curve and $f \in I^p \setminus I^{p+1}$, then $ht(Jac(f))=1$, and thus $f$ is $\mathcal{R}-$equivalent to an element of $\mathbb{C}[x_1, \dots, x_{n-1}]\{x_n\}$ (See Proposition 5.19 in~\cite{boix2018pairs}). Moreover, if $f_t$ is a relative morsification and if the automorphism $\varphi \colon \mathbb{C}\{\underline{x}\} \to \mathbb{C}\{\underline{x}\}$ preserves the singular locus of $f_t$, then $\varphi^*(f_t)=f_t \circ \varphi$ is a relative morsification as well. Thus,  it is enough to prove the existence of a relative morsification for analytic germs of the form $f \in \mathbb{C}\{x_n\}[x_1, \dots, x_{n-1}]$.\\

We view $f$ as an analytic function $f \colon U \to \mathbb{C}$, where $ 0 \in U \subset \mathbb{C}^n$ is an open cylinder $U=\mathbb{C}^{n-1} \times U_s$ with $U_s \subset \mathbb{C}^1_{x_1}$. For every $s_0 \in U_s$, denote by $f_{s_0}$ the tuple corresponding  with the polynomial in $(n-1)$ variables $f(x_1, \dots, x_{n-1}, s_0)$. (Here we identify $\mathbb{C}[x_1, \dots, x_{n-1}]_{\leq d_{n-1}}$ with the affine space $ \mathbb{C}^{K}$ where $K=\binom{d_{n-1}+n-1}{n-1}$ and $d_{n-1}$ is the degree of $f$ with respect to $x_1, \dots, x_{n-1}$). Denote by $\Gamma$ the set $\{f_s \colon s \in U_s\} \subset \mathbb{C}^K$. Since $f \in \mathbb{C}\{x_n\}[x_1, \dots, x_{n-1}]$ then $(\Gamma,0)$ is an irreducible analytic curve-germ in $\mathbb{C}^{K}$.\\

Now, note that $f$ is Morse outside $V(I)$ (as in Definition~\ref{def:morsification}) if and only if for every small enough $s$, the polynomial  $f_s \in \mathbb{C}[x_1, \dots, x_{n-1}]$ has no non-Morse points which are not the origin.  Denote 
\begin{equation*}
    \Sigma_0=\{g \in \langle x_1, \dots x_{n-1} \rangle^p \subset \mathbb{C}[x_1, \dots, x_{n-1}] \colon g \text{ is not Morse outside the origin} \}.
\end{equation*}
As in Part 1.1, we have that $\overline{\Sigma_0} \subset  \mathbb{C}^{K}$ is an algebraic subset. Since $(\Gamma,0)$ is an irreducible analytic curve, then either $\Gamma \cap \overline{\Sigma_0}$ is a discrete set or $\Gamma \subset \overline{\Sigma_0}$. \\

If $\Gamma \cap \overline{\Sigma_0}$ is a discrete set, then we can choose a neighborhood $\tilde{U}_{s} \subset U_s$ such that $\{ f_s \colon s \in \tilde{U}_s\} \cap \overline{\Sigma_0}$ is a single point (where $s=0$) and thus $f$ is Morse outside $V(I)$ in $\mathbb{C}^{n-1} \times U_s$.\\

Otherwise, $\Gamma \subset \overline{\Sigma_0}$, and since $\overline{\Sigma_0}$ is an algebraic subset then there exists some $g \in \langle x_1, \dots, x_{n-1}\rangle^p \subset \mathbb{C}[x_1, \dots, x_{n-1}]$ such that, for every $s$, $tg + (1-t)f_s$ intersects $\overline{\Sigma_0}$ only when $t=0$, and therefore $tg+(1-t)f$ is Morse outside $V(I)$. \\

By Part 1.2, since $\Sigma_2 \subset \mathbb{C}[\underline{x}] \subset \mathbb{C}\{x_n\}[x_1, \dots, x_{n-1}]$ is an algebraic subset, we can choose $g$ such that $tg+(1-t)f$ is a relative morsification of $f$. 
 
\end{proof}


We end this section by presenting a sketch of a more analytic proof of Theorem~\ref{prop:morsif}.  This proof is similar to the proof of Theorem 1.15 in~\cite{de1988some} and to the proof of Proposition 7.18 in~\cite{pellikaanhypersurface}, and heavily relies on tools from differential geometry such as Sard's theorem and Fubini's theorem. In addition, it gives us a more explicit form for the relative morsification (similar to the relative morsifications in Example~\ref{ex:def_mors_stuff}). 

 \begin{proof}

\textup{Consider $f$ as analytic function $f \colon U \to \mathbb{C}$ where $U \subset \mathbb{C}^n$. Set $N=(n-1)^2+n(n-1)$ and define $F \colon U \times \mathbb{C}^N \to \mathbb{C}$ by} 
\begin{equation*}
    F(\underline{x}, \underline{a}, \underline{b})=f(x)-\big(\sum_{i, j \leq n-1 } a_{i,j} x_j^{p-2} x_i^2 + \sum_{ k\neq l} b_{k,l} x_k x_l^p\big).
\end{equation*}

\textup{ We construct from $F(\underline{x}, \underline{a}, \underline{b})$ a relative morsification of $f$. Note that for every small enough $(\underline{a}, \underline{b})$ we have that the singular locus of $F(\underline{x}, \underline{a}, \underline{b})$ (with respect to $\underline{x}$) is exactly $V(I)$.  }\\

First, denote $F_{\underline{a}, \underline{b}}(\underline{x})= F(\underline{x}, \underline{a}, \underline{b})$ and denote by $h_{\underline{a}, \underline{b}} \in \mathbb{C}\{x_n\}$ the generating element of $I(\Delta^\perp(F_{\underline{a}, \underline{b}}))$. Let $H=\{(x_n, \underline{a}, \underline{b}) \in (V(I) \cap U) \times \mathbb{C}^N \colon h_{\underline{a}, \underline{b}}(x_n)=0\}$ where we identify $V(I)$ with a copy of $\mathbb{C}$. By directly computing the transversal discriminant of $F_{\underline{a}, \underline{b}}$ one can verify that $h_{\underline{a}, \underline{b}} $ is not the zero function, for every  $(\underline{a}, \underline{b}) \in \mathbb{C}^N$. Thus the projection map $p \colon H \to \mathbb{C}^N$ is a finite map since $h_{\underline{a}, \underline{b}}(x_n)$ has finitely many zeros in $V(I) \cap U$. By Sard's Theorem, the set $S_1 = p(\Crit(p))$ is a set of measure zero, and outside $p^{-1}(S_1)$ we have that $p$ is a finite covering map. Thus for every $\underline{s} \notin S_1$ we have that $\Delta^\perp (F_{\underline{s}})(x_n)$ is reduced.\\ 

\textup{Second, denote $\Phi_i=\frac{\partial F}{\partial x_i}$ and consider the function $\Phi=(\Phi_1, \dots, \Phi_n) \colon U \times \mathbb{C}^N \to \mathbb{C}^n$. Note that in order to prove that $F(x,\underline{a}, \underline{b})$ is a Morse function outside $V(I)$, it is enough to show that $\Phi(x,\underline{a}, \underline{b})$ is a submersion outside $(V(I) \cap U) \times \mathbb{C}^N$. Because this would imply that $\det(\Hess(F))\neq 0$. Accordingly, we study the critical locus of $\Phi$ via the differential $d\Phi \in \text{M}_{n \times (n+N)}(\mathbb{C})$. One calculates that for every $i,j$ we have that }
\begin{equation*}
    \frac{\partial \Phi_i}{ \partial a_{j,j}}= p x_j^{p-1} \delta_{i,j}
\end{equation*}
\textup{and for every $j \neq k$ and for every $i$ we have that} 
\begin{equation*}
    \frac{\partial \Phi_i}{ \partial b_{j,k}}=  x_k^{p-1} (p x_i \delta_{k,i} +\delta_{i,j} x_k).
\end{equation*}

\textup{ This gives us that }
\begin{equation*}
    \det \Bigg[ \frac{\partial \Phi}{\partial b_{1,k}} - x_1 \frac{\partial \Phi}{\partial a_{k,k}}, \dots, \frac{\partial \Phi}{\partial b_{n,k}} - x_n \frac{\partial \Phi}{\partial a_{k,k}} \Bigg] = x_i ^{p(n+1)}.
\end{equation*}

\textup{Thus, the ideal $\mathfrak{a}=\langle x_1^{p(n+1)}, \dots,  x_{n-1}^{p(n+1)}\rangle$ is contained in the ideal generated by the $n\times n$ minors of $d\Phi$. So $V(\mathfrak{a})=V(I)$ contains the critical locus of $\Phi$, and thus $\Phi$ is a submersion outside $(V(I) \cap U) \times \mathbb{C}^N$.}\\

\textup{Note that $\Phi^{-1}(0) \setminus ((V(I) \cap U) \times \mathbb{C}^N)$ is smooth of dimension $N$. therefore if we look at the projection $\pi \colon \Phi^{-1}(0) \setminus ((V(I) \cap U) \times \mathbb{C}^N) \to \mathbb{C}^N$, then this projection is non degenerate. }\\
 
 \textup{Now, define $\Psi \colon (U \setminus V(I)) \times \mathbb{C}^N \to Gr(\mathbb{C}^n \times \mathbb{C}^N,N)$ by $\Psi(x) = \ker(d\Phi(x))$, where $Gr(\mathbb{C}^n \times \mathbb{C}^N,N)$ is the $N-$th Grassmannian of $\mathbb{C}^n \times \mathbb{C}^N$. Note that $\Psi$ is $C^\infty$ and its critical local is of measure zero. Observe that the set $\mathcal{H} = \{ V \in Gr(\mathbb{C}^n \times \mathbb{C}^N,N) \colon V \cap (\mathbb{C}^n \times \{0\}) \neq \{0\} \}$ is a compact space of measure zero, and thus $\Psi^{-1}(\mathcal{H}) \subset (U \setminus V(I)) \times \mathbb{C}^N$ is closed and of measure zero as well. Thus, by Fubini's theorem, almost every section of $\Psi^{-1}(\mathcal{H})$ with respect to $\mathbb{C}^N$ is closed and of measue zero as well, and so its complement is open and of full measure.  }\\
 
 \textup{Therefore, for every $\underline{s} \notin S_2$ we have that $\Phi(\cdot, \underline{s}) \colon U \to \mathbb{C}^n$ is a submersion, i.e., $F(\underline{x}, \underline{s})$ has only Morse critical points in $U_1$ and outside $V(I)$, where $S_2$ is the set of elements in $\mathbb{C}^N$ with respect to which the sections of $\Psi^{-1}(\mathcal{H})$ are not of measure zero (up to shrinking $U$).}\\

Since $S_1$ and $S_2$ are sets of measure zero, then there exists some $s=(\underline{a}, \underline{b}) \in \mathbb{C}^N$ such that for every $0 \neq t \in \mathbb{C}$ small enough, we have that $t \cdot  s \notin S_1 \cup S_2$. Therefore we can conclude that $f_t(\underline{x})= F(\underline{x}, {t \cdot s})$ is a relative morsification of $f$.


\end{proof}

\section{Siersma-Pellikaan-de Jong Formula}\label{section:SPJ}

In this section we discuss how, inspired by Siersma, Pellikaan, and de Jong, we can bound the number of Morse points in any relative morsification and the degree of the transversal discriminant of $f$ (as a Cartier divisor) by an algebraic invariant. We start with a general construction (Theorem~\ref{thm:main}) from which we give a more concrete and computational invariant (Theorem~\ref{thm:main_Jac_gen}). \\

Recall that we assume that $f \in I^p \setminus I^{p+1}$ for some $p\geq 2$ with $\text{Sing}(V(f)=V(I)$ such that its generic transversal type is an ordinary multiple point. In addition, throughout this section we assume that $n \neq 4$, as it is an important assumption in Proposition~\ref{prop:trans_disc_LR} and in its application in the proof of Lemma~\ref{lemma:j=0_and_disc} (see Remark~\ref{rem:Le} for more details). 

\begin{notation}\label{notation:white_mountain}
\textup{Let $f \in I^p \setminus I^{p+1}$ with $\text{Sing}(V(f))=V(I)$ and let $f_t$ be a deformation of $f$. Consider the following objects:}
\begin{enumerate}
\item  \textup{$R=\mathbb{C}\{\underline{x},t\}$ is a local ring with a unique maximal ideal $\mathfrak{m} = \langle x_1, \dots, x_n , t \rangle$.}
\item \textup{The ideal $\mathfrak{p}=\langle x_1, \dots, x_{n-1} \rangle \subset R$.}
    \item \textup{$\# A_1 (f_{t_0})$ is the number of Morse points of a  $f_t$ outside $V(I)$, for a small and fixed $t_0$.}
    \item \textup{The ideal $Jac(f_t)=\langle \partial_1(f_t), \dots, \partial_n(f_t) \rangle \subset \mathbb{C}\{\underline{x},t\}$, where $\partial_i(f_t)$ is the partial derivative of $f_t$ with respect to $x_i$.  }
    \item\textup{The module $M(\mathfrak{a})=\sfrac{\mathfrak{a}}{Jac(f_t)}$,
    and $j_0(\mathfrak{a})=\dim_\mathbb{C} ( M(\mathfrak{a}) \otimes_{\mathbb{C}\{\underline{x},t\}} \sfrac{\mathbb{C}\{\underline{x},t\}}{\langle t \rangle } ) $, for every ideal $Jac(f_t) \subset \mathfrak{a}$. }
\end{enumerate}
\end{notation}

\begin{remark}\label{remark:jac_gen}
    
 \textup{Note that we have an isomorphism}
    \begin{equation*}
        (M(\mathfrak{a}) \otimes_{\mathbb{C}\{\underline{x},t\}} \sfrac{\mathbb{C}\{\underline{x},t\}}{\langle t \rangle } ) = \sfrac{\mathfrak{a}}{Jac(f_t) + \langle t \rangle \cdot \mathfrak{a}}. 
    \end{equation*}

\end{remark}

\begin{definition}
Let $f_t$ be a deformation of $f$. An ideal $Jac(f_t) \subset \mathfrak{a}$ is called a \textbf{generification of $Jac(f_t)$ over $\mathfrak{p}$} if 
\begin{enumerate}
    \item The ring $\sfrac{R}{\mathfrak{a}}$ is a Cohen-Macaulay ring of dimension $2$ such that  $t$ is not a zero-divisor of $\sfrac{R}{\mathfrak{a}}$.
    \item $\sqrt{\mathfrak{a}}=\mathfrak{p}$ such that $Jac(f_t)_\mathfrak{p} = \mathfrak{a}_\mathfrak{p}$. 
\end{enumerate}
\end{definition}


\begin{theorem}\label{thm:main}
Let $f_t$ be a relative morsification of $f$, and let $\mathfrak{a}$ be a generification of $Jac(f_t)$ over $\mathfrak{p}$. Then for every small $t_0$, 
\begin{equation*}
    j_0(\mathfrak{a})\geq  \# A_1 (f_{t_0}) +\deg (\Delta^\perp(f)). 
\end{equation*}
\end{theorem}

Before proving Theorem~\ref{thm:main}, we start with a few lemmata and a few additional remarks and notations, some of which are quotes of known results from the literature. 

\begin{notation}\label{def:all_hail_his_satanic_majesty}

\textup{Let $f_t$ be a relative morsification of $f$ and let  $ \mathfrak{a}$ be a generification of $Jac(f_t)$ over $\mathfrak{p}$. Consider $F(\underline{x},t)=f_t(\underline{x})$ as an analytic function $F \colon U \to \mathbb{C}$ where $U \subset \mathbb{C}^{n+1}$ is a bounded small neighborhood of $0$ of the form $U =U_n \times U_1$, where $U_n \subset \mathbb{C}^n$ and $U_1 \subset \mathbb{C}^1$. Then for every $(p, t_0) \in U$ we define }
\begin{equation*}
    j(\mathfrak{a}, p, t_0)=\dim_\mathbb{C}(\pi_*(\mathcal{M}_{(p,t_0)})),
\end{equation*}
\textup{where $\mathcal{M}=\sfrac{\mathcal{I}_F}{\mathcal{J}_F}$, in which $\mathcal{J}_F$ is the ideal sheaf generated by $\frac{\partial F}{\partial x_1} \dots, \frac{\partial F}{\partial x_n}$, $\mathcal{I}_F=\mathfrak{a} \cdot \mathcal{O}_U$, and $\pi \colon (U \cap \text{Supp}(\mathcal{M})) \to U_1$ is the projection map.}
\end{notation}

\begin{remark}
    \textup{As with $j_0(\mathfrak{a})$, the invariant $j(\mathfrak{a},p, t_0)$ might be infinite. We address the finiteness of $j_0(\mathfrak{a})$ in Lemma~\ref{lemma:j_0_finite}, and of $j(\mathfrak{a},p, t_0)$ in Lemma~\ref{lemma:j_additive}. Yet in general, since we are looking at $\pi_*(\mathcal{M}_{(p,t_0)})$, we are looking at the push-forward of a one-dimensional stalk onto a one dimensional ring, $\mathbb{C}\{t\}$. }
\end{remark}

The following proposition, Proposition~\ref{prop:quotient}, is a restatement of a result which was originally presented as Proposition 1.17 in Part 2 of~\cite{pellikaanhypersurface}. This proposition plays a crucial role in the proof of Theorem~\ref{thm:main}, specifically as part of the proof of Lemma~\ref{lemma:M(f_t)_CM}. Originally it was presented in more generality by Artin and Nagata as Theorem 2.1 in~\cite{artin1972residual}, yet a counterexample to their general case and a correction has been presented by Huneke~\cite{huneke1983strongly} as Theorem 3.1, and the proof we present below was inspired by their results. \\

For more information on Cohen-Macaulay modules, regular sequences, and depth, see Section 17 of~\cite{eisenbud2013commutative}.  (Note that since our ring is local, by depth we mean depth of a module with respect to the corresponding maximal ideal.)


\begin{proposition}\label{prop:quotient}
Let $(S, \mathfrak{n})$ be a regular local ring and let $\mathfrak{a},\mathfrak{b} \subset S$ be ideals such that
\begin{enumerate}
    \item $\sfrac{S}{\mathfrak{a}}$ is a Cohen-Macaulay ring with $\dim(\sfrac{S}{\mathfrak{a}})=\dim(S)-r$,
    \item $\mathfrak{b}=\langle b_1, \dots, b_{r+1} \rangle \subsetneq \mathfrak{a}$,
    \item $\mathfrak{a}_\mathfrak{q}=\mathfrak{b}_\mathfrak{q}$ for every prime ideal $\mathfrak{q}$ such that $\dim(\sfrac{S}{\mathfrak{q}}) > \dim(S)-(r+1)$. 
\end{enumerate}
Then $\sfrac{\mathfrak{a}}{\mathfrak{b}}$ is a Cohen-Macaulay module over $S$, with
\begin{equation*}
    \dim (\text{Supp}(\sfrac{\mathfrak{a}}{\mathfrak{b}}))=\dim(S) - (r+1).  
\end{equation*}

\end{proposition}

\begin{proof}

First, since $\mathfrak{a}_\mathfrak{q}=\mathfrak{b}_\mathfrak{q}$ for every prime ideal $\mathfrak{q}$ such that $\dim(\sfrac{S}{\mathfrak{q}}) > \dim(S)-(r+1)$, then for every such prime ideal $\mathfrak{q}$ we have that $(\sfrac{\mathfrak{a}}{\mathfrak{b}})_{\mathfrak{q}}=0$. Therefore $\dim (\text{Supp}(\sfrac{\mathfrak{a}}{\mathfrak{b}})) \leq \dim(S) - (r+1)$, and so we would like the prove that the depth of $\sfrac{\mathfrak{a}}{\mathfrak{b}}$ is at least  $\dim(S) - (r+1)$. But, by looking at the exact sequence  $0 \to \sfrac{\mathfrak{a}}{\mathfrak{b}} \to \sfrac{S}{\mathfrak{b}} \to \sfrac{S}{\mathfrak{a}} \to 0$, we can conclude that it is enough to show that the depth of $\sfrac{S}{\mathfrak{b}}$ is at least $\dim(S) - (r+1)$. \\

Now, if $\mathfrak{b}$ is a complete intersection (i.e. can be generated by $r$ elements in $S$), then we would get that $\sfrac{S}{\mathfrak{b}}$ is a Cohen-Macaulay ring, and so $\text{depth}(\sfrac{S}{\mathfrak{b}}) =\dim(\sfrac{S}{\mathfrak{b}}) = \dim(S)-r$. Otherwise, since $\mathfrak{b}$ is generated by $r+1$ elements, namely $b_1, \dots, b_{r+1}$, we can assume without loss of generality that $b_1, \dots, b_r$ is an $S-$regular sequence. This is true since $S$ is regular and so the grade of $S$ over $\mathfrak{b}$ is exactly $ht(\mathfrak{b})=r$ (see Section 1.2 in~\cite{bruns1998cohen} or the introduction to Chapter 2 of~\cite{pellikaanhypersurface}). Denote $\mathfrak{c}=\langle b_1, \dots, b_r \rangle$. Since $\dim(\sfrac{S}{\mathfrak{a}})=\dim(S)-r$ and $\mathfrak{c} \subset \mathfrak{a}$, then $b_1, \dots, b_r$ is a maximal $S-$regular sequence in $\mathfrak{a}$. So, from Propositions 1.6 and 1.7 in~\cite{huneke1983strongly} we have that $\sfrac{S}{(\mathfrak{c} \colon \mathfrak{a})}$ is a Cohen-Macaulay ring of dimension $\dim(S)-r$ and $b_{r+1}$ is a non zero-divisor of $\sfrac{S}{(\mathfrak{c} \colon \mathfrak{a})}$. \\

Thus, if we look at the exact sequence of $S-$modules $0\to \mathfrak{c} \to S \to \sfrac{S}{\mathfrak{c}} \to 0$, since $\mathfrak{c}$ is generated by an $S-$regular of length $r$, then we get that $\text{depth}(\mathfrak{c}) \geq \dim(S) -r+1$ (where we view $\mathfrak{c}$ as an $S-$module). On  the other hand, we have the isomorphisms 
\begin{equation*}
    \sfrac{\mathfrak{b}}{\mathfrak{c}} \cong \sfrac{\langle b_{r+1} \rangle }{\langle b_{r+1} \rangle \cap \mathfrak{c}} \cong \sfrac{S}{(\mathfrak{c} \colon \langle b_{r+1} \rangle )} = \sfrac{S}{(\mathfrak{c} \colon \mathfrak{b} )}.  
\end{equation*}

Since $(\mathfrak{c} \colon \mathfrak{a}) \subset (\mathfrak{c} \colon \mathfrak{b})$, in addition to  $b_{r+1} (\mathfrak{c} \colon \langle b_{r+1} \rangle )=b_{r+1} (\mathfrak{c} \colon \mathfrak{b}) \subset \mathfrak{c} \subset (\mathfrak{c} \colon \mathfrak{a})$ and that $b_{r+1}$ is a non zero-divisor of $\sfrac{S}{(\mathfrak{c} \colon \mathfrak{a})}$, we can conclude that $(\mathfrak{c} \colon \mathfrak{a}) = (\mathfrak{c} \colon \mathfrak{b})$. Thus $\sfrac{\mathfrak{b}}{\mathfrak{c}} \cong \sfrac{S}{(\mathfrak{c} \colon \mathfrak{a})}$, which has depth $\dim(S)-r$. Therefore, if we look at the exact sequence $0 \to \mathfrak{c} \to \mathfrak{b} \to \sfrac{\mathfrak{b}}{\mathfrak{c}} \to 0$, we can conclude that $\text{depth}(\mathfrak{b}) \geq \dim(S) -r$ (where we view $\mathfrak{b}$ as an $S-$module). Thus, by looking at the exact sequence $0\to \mathfrak{b} \to S \to \sfrac{S}{\mathfrak{b}} \to 0$, we can conclude that $\dim(S) -(r+1) \leq \text{depth}(\sfrac{S}{\mathfrak{b}})$ and the result follows. 

\end{proof}

\begin{remark}\label{remark:boom_boom_paw}
\begin{enumerate}
    \item \textup{In Proposition~\ref{prop:quotient}, since $S$ is a regular local ring, if we look at the exact sequence of modules $0 \to \mathfrak{b} \to \mathfrak{a} \to \sfrac{\mathfrak{a}}{\mathfrak{b}} \to 0$ and recall that localization is an exact functor, then condition $3$ is in fact equivalent to the condition that $\dim (\text{Supp}(\sfrac{\mathfrak{a}}{\mathfrak{b}})) \leq \dim(S) - (r+1)$. This gives us an alternative proof of Proposition 3.3 in~\cite{pellikaan1988projective} (in the case where $S$ is a regular local ring). }
    \item \textup{Since $\mathfrak{b}$ can be generated by $r+1$ and $\dim(\sfrac{S}{\mathfrak{b}})=\dim(S)-r$, then if $\mathfrak{b}$ is not a complete intersection, it is an almost complete intersection. The connection between almost complete intersections and non-isolated singularities (in the real case, in addition to their connection to Pellikaan's work) have been studied by van Straten and Warmt in~\cite{van2015gorenstein}. }
\end{enumerate}

\end{remark}

\begin{lemma}\label{lemma:M(f_t)_CM}
Let $f_t$ be a relative morsification of $f$, and let  $ \mathfrak{a}$ be a generification of $Jac(f_t)$ over $\mathfrak{p}$. Then  $M(\mathfrak{a})$ is a Cohen-Macauly module over $R$. 
\end{lemma}

\begin{proof}
Let $f \in I^p \setminus I^{p+1}$ and let $f_{t}=f+tg$ be a relative morsification of $f$ (which exists by Theorem~\ref{prop:morsif}). If $M(\mathfrak{a}) =0$ we are done, so we can assume that $M(\mathfrak{a})\neq 0$. We first show that $\dim(\text{Supp}(M(\mathfrak{a}))) \leq 1$. \\

First, since we can write $f_t =f+tg$ where $f,g \in \mathbb{C}\{\underline{x}\}$, we have that $Jac(f_t) +\langle t \rangle =Jac(f) \cdot R  +\langle t \rangle$. Therefore, we have that 
\begin{equation*}
    \sfrac{R}{Jac(f_t)} \otimes\sfrac{R}{\langle t \rangle } =\sfrac{R}{Jac(f) \cdot R  +\langle t \rangle} \cong \sfrac{\mathbb{C}\{\underline{x}\}}{Jac(f)},
\end{equation*}

where $Jac(f)=\langle \partial_1(f), \dots, \partial_n(f) \rangle \subset \mathbb{C}\{\underline{x}\}$, which gives us that

\begin{equation*}
     \dim(\sfrac{R}{Jac(f_t)}) \leq \dim(\sfrac{\mathbb{C}\{\underline{x}\}}{Jac(f)}) +1 =2, 
\end{equation*}

 as $V(Jac(f)) =\Sing(V(f))$. Yet, because $f_t$ is a morsification, we have that $Jac(f_t) \subset \mathfrak{p}$ and so we can conclude that $\dim(\sfrac{R}{Jac(f_t)}) \geq \dim(\sfrac{R}{\mathfrak{p}}) =2$. Therefore we have that $ \dim(\sfrac{R}{Jac(f_t)})=2$. Now, from the exact sequence $0 \to M(\mathfrak{a}) \to \sfrac{R}{Jac(f_t)} \to \sfrac{R}{\mathfrak{a}}\to 0$ we can conclude that $\dim(\text{Supp}(M(f_t))) \leq  \dim(\text{Supp}(\sfrac{R}{Jac(f_t)})) = 2$, and since we know that $Jac(f_t)_\mathfrak{p}=\mathfrak{a}_\mathfrak{p}$, we get that $\mathfrak{p} \notin \text{Supp}(M(\mathfrak{a}))$. \\

 Following Notation~\ref{def:all_hail_his_satanic_majesty}, consider the quotient sheaf $\sfrac{\mathcal{O}_U}{\mathcal{J}_F}$ and the sheaf $\mathcal{M}$. Note that $\mathcal{M}_0 = M(\mathfrak{a})$ and $(\sfrac{\mathcal{O}_U}{\mathcal{J}_F})_0 = \sfrac{R}{Jac(f_t)}$. It is enough to show that $\dim(\text{Supp}(\mathcal{M})) \leq 1$ as a complex germ, since we have that  $\dim(\text{Supp}(M(\mathfrak{a}))) = \dim(\text{Supp} (\mathcal{M}_0)) \leq  \dim(\text{Supp}(\mathcal{M}))$. \\
 
 Since $V(\mathfrak{p}) \subset V(Jac(f_t))$ we have that $V(x_1, \dots, x_{n-1}) \subset \text{Supp}(\sfrac{\mathcal{O}_U}{\mathcal{J}_F})$. Therefore, because we have that $\text{Supp}(\mathcal{M})$ is a closed subset of $\text{Supp}(\sfrac{\mathcal{O}_U}{\mathcal{J}_F})$ with $\dim( \text{Supp}(\sfrac{\mathcal{O}_U}{\mathcal{J}_F})) =2$,  we need to prove that $\text{Supp}(\sfrac{\mathcal{O}_U}{\mathcal{J}_F})$ contains no irreducible components of dimension $2$ except for $V(x_1, \dots, x_{n-1})$.\\
 
 Yet, since $f_t$ is a relative morsification of $f_t$, then for every small enough $t_0$ we have that $(\Crit(f_{t_0})  \cap U_n) \setminus (V(x_1, \dots, x_{n-1}))$ is composed of a finite set of discrete points (which are exactly the Morse points of $f_{t_0}$). Thus for every small enough $t_0$ we have that $( \text{Supp}(\sfrac{\mathcal{O}_U}{\mathcal{J}_F}|_{\pi^{-1}(t_0)}) \cap U_n) \setminus (V(x_1, \dots, x_{n-1}))$ is zero dimensional and from Exercise 1.3.2 in~\cite{greuel2007introduction} we can conclude that indeed $\text{Supp}(\sfrac{\mathcal{O}_U}{\mathcal{J}_F})$ contains no irreducible components of dimension $2$ except for $V(x_1, \dots, x_{n-1})$.\\

Now, since $\dim(\text{Supp}(M(\mathfrak{a}))) \leq 1$, $\sfrac{R}{\mathfrak{a}}$ is a Cohen-Macaulay ring of dimension $2$, and $Jac(f_{t})=\langle \partial_1(f_{t}), \dots, \partial_n(f_{t}) \rangle$, we can apply Proposition~\ref{prop:quotient} and the first part of Remark~\ref{remark:boom_boom_paw} where $\mathfrak{b}=Jac(f_t)$, $\dim(R)=n+1$, and $r=n-1$, and get that $M(\mathfrak{a})$ is Cohen-Macaulay over $R$ with $\dim(\text{Supp}_R(M(\mathfrak{a})))=1$, as desired. \\

\end{proof}

\begin{remark}\label{rem:24}
\begin{enumerate}
    \item \textup{Looking at the second part of the proof of Proposition~\ref{lemma:M(f_t)_CM} from a geometric point of view, the irreducible decomposition of $V(Jac(f_t))$ is of the form $V(\mathfrak{p}) \cup V(\mathfrak{q}_1) \cup \cdots \cup V(\mathfrak{q}_s)$ where $V(\mathfrak{q}_i)$ is a curve for every $i$ and $s$ is exactly the number of $A_1$ points in this morsification. This statement is elaborated in the proof of Theorem~\ref{thm:main}.}
    \item  \textup{In Lemma~\ref{lemma:M(f_t)_CM}, the fact that $f_t$ is a relative morsification plays a crucial role here. For example, $f_t(x,y,z)=x^{p+1}+y^p+tx^p$ is not a relative morsification of $f(x,y,z)=x^{p+1}+y^p$ (as we saw in Example~\ref{ex:def_mors_stuff}) and $\text{Supp}(\sfrac{\mathcal{O}_U}{\mathcal{J}_F})$ has more than one irreducible component of dimension $2$.}
\end{enumerate}

\end{remark}

The following proposition, Proposition~\ref{prop:bigbigbig},  plays a crucial role in the proof of Theorem~\ref{thm:main}, Theorem~\ref{thm:delta_trans_bound}, and other related results. It is a restatement of Proposition 7.2 from
\cite{pellikaanhypersurface} and of Theorem 6.4.7 in~\cite{de2013local}.

\begin{proposition}\label{prop:bigbigbig}
Let $U \subset \mathbb{C}^{n+1}$ be a small open neighborhood of the origin of the form $U=U_1 \times U_n$ where $U_1 \subset\mathbb{C}_t$ and $U_n \subset \mathbb{C}^n_{\underline{x}}$ are both small open neighborhood of their respect origins. Denote by $\pi \colon U \to U_1$ be the projection map. Let $\mathcal{N}$ be a coherent $\mathcal{O}_U$-module such that $\mathcal{N}_0$, the stalk of $\mathcal{N}$ at the origin, is Cohen-Macaulay over $\mathbb{C}\{\underline{x},t\}$ of dimension $1$ and of finite rank over $\mathbb{C}\{t\}$. Then for every small $t_0$ we have that 
\begin{equation*}
    \dim_\mathbb{C} ((\mathcal{N}|_{\pi^{-1}(0)})_0)=\sum_{s \in \pi^{-1}(t_0)} \dim_\mathbb{C} ((\mathcal{N}|_{\pi^{-1}(t_0)})_p).
\end{equation*}

Here the fibre of the module is defined via the restriction with respect to $U_1$, that is, $\mathcal{N}|_{\pi^{-1}(t_0)}:= \mathcal{N}|_{\pi^{-1}(t_0)} \otimes \sfrac{\mathcal{O}_U}{ \langle t-t_0 \rangle}$ for $t_0 \in U_1$. 

\end{proposition}

\begin{remark}\label{rem:Le}
    \textup{The following result, Proposition~\ref{prop:trans_disc_LR}, relies heavily on the L\^{e}-Ramanujan theorem, which is a result from~\cite{trang1976invariance}, and plays a crucial role in understanding how the hypersurface with an empty transversal discriminant behaves. This result is an important component in the proof of Theorem~\ref{thm:main}, and especially in the proof of Lemma~\ref{lemma:j=0_and_disc}. As the L\^{e}-Ramanujan theorem is still open in dimension $3$, we assume that $n-1 \neq 3$. For a deeper discussion on this theorem and related results, see article titled "Equisingularity" in~\cite{dos2018singularities}. }
\end{remark}

\begin{proposition}\label{prop:trans_disc_LR}
If $\sfrac{\mathbb{C}\{\underline{x}\}}{Jac(f)}$ is Cohen-Macaulay over $\mathbb{C}\{\underline{x}\}$ then $\Delta^\perp(f)=\emptyset$. 
\end{proposition}

\begin{proof}
Following the convention of Notation~\ref{def:all_hail_his_satanic_majesty}, we consider $f$ as an analytic function $f \colon U \to \mathbb{C}$, where $U \subset \mathbb{C}^n$ is a small bounded open set of the form $U=U_{n-1} \times U_1$ where $U_{n-1} \subset \mathbb{C}^{n-1}$ and $U_1 \subset V(I)$, where we view $V(I)$ as a copy of $\mathbb{C}^1$. Consider the quotient sheaf $\mathcal{N}=\sfrac{\mathcal{O}_U}{\mathcal{J}}$, where $\mathcal{J}$ is the sheaf generated by $\frac{\partial f}{\partial x_1}, \dots, \frac{\partial f}{\partial x_n}$, and let $\pi \colon U \to U_1$ be the projection map. \\

Since $(\sfrac{\mathcal{O}_U}{\mathcal{J}})_0 = \sfrac{\mathbb{C}\{\underline{x}\}}{Jac(f)}$ is a Cohen-Macaulay module, then by Proposition~\ref{prop:bigbigbig} we have that, up to shrinking $U$, the stalks of $(\pi_* \mathcal{N})|_{\tilde{U}_1}$ are free over $\mathbb{C}\{x_n\}$ and of the same rank. Note that for every $z_0 \in U_1$ we have that as module over $\mathbb{C}\{x_n\}$,
\begin{equation*}
    ((\mathcal{N}|_{\pi^{-1}(z_0)})_{(0, \dots, 0, z_0)} \cong \sfrac{\mathbb{C}\{x_1, \dots, x_{n-1}\}}{Jac(f_{z_0})},
\end{equation*}

 where $f_{z_0}(x_1, \dots, x_{n-1})=f(x_1, \dots, x_{n-1}, z_0) \in \mathbb{C}\{x_1, \dots, x_{n-1}\}$ and where $Jac(f_{x_0})$ is the ideal generated by the derivatives of $f_{z_0}$ with respect to the variables $x_1, \dots, x_{n-1}$. We can conclude that for every $z_1,z_2 \in U_1$,
\begin{equation*}
    \mu(f_{z_1}) = \dim_\mathbb{C} (\pi_* ((\mathcal{N}|_{\pi^{-1}(z_1)})_{(0, \dots, 0, z_1)})=\dim_\mathbb{C} ((\mathcal{N}|_{\pi^{-1}(z_2)})_{(0, \dots, 0, z_2)}) =\mu(f_{z_2}).
\end{equation*}
 Thus the Milnor number of $f_{z_0}$ does not on the choice of $z_0 \in {U}_1$, and the result follows from the L\^{e}-Ramanujan theorem (recall that $n-1 \neq 3$) and Remark~\ref{rem:top_trans_disc}. 
\end{proof}

\begin{remark}\label{rem:counterex}
 \textup{The converse of of Proposition~\ref{prop:trans_disc_LR} is not true. For example, take $f(x,y,z)=2k(x^{3k} +y^{3k}) -3k z^2 x^{2k} y^{2k}$ where $k \geq 2$. Then  $\Delta^\perp(f)=\emptyset$, but $S=\sfrac{\mathbb{C}\{x,y,z\}}{Jac(f)}$ is not a Cohen-Macaulay module, since }
\begin{equation*}
    Jac(f) = \langle x^{3k-1}-z^2x^{2k-1}y^{2k}, y^{3k-1} -z^2 x^{2k}y^{2k-1}, z x^{2k} y^{2k} \rangle
\end{equation*}

\textup{and by direct computation we get that $\dim(S)=1$ but $\text{depth}(S)=0$, as both $x,y$ and $z$ are all zero divisors of $S$.}

\end{remark}

\begin{lemma}\label{lemma:j=0_and_disc}
Let $f \in I^p \setminus I^{p+1}$ with $\Sing(V(f))=V(I)$, and let $\mathfrak{a}$ be a generification of $Jac(f_t)$ over $\mathfrak{p}$. 
\begin{enumerate}
    \item For every relative morsification of $f$,  $j_0(\mathfrak{a})=0$ if and only if $M(\mathfrak{a}) = 0$.
    \item If there exists some relative morsification $f_t$ of $f$ such that $j_0(\mathfrak{a})=0$, then $\Delta^\perp(f)=\emptyset$. 
\end{enumerate}
\end{lemma}

\begin{proof}
\begin{enumerate}
    \item Let $f_t$ be a relative morsification of $f$. First, if $M(\mathfrak{a}) = 0$ then obviously $j_0(\mathfrak{a})=0$. Second, assume that $j_0(\mathfrak{a})=0$. Then $M(\mathfrak{a}) \otimes_{R} \sfrac{R}{\langle t \rangle}=0$, and thus $\text{Supp}(M(\mathfrak{a})) \cap V(\langle t \rangle) = \emptyset$. We can conclude that $M(\mathfrak{a})=0$, since otherwise, by Lemma~\ref{lemma:M(f_t)_CM}, we would have that $\dim(\text{Supp}(M(\mathfrak{a})))=1$, and thus the intersection $\text{Supp}(M(\mathfrak{a})) \cap V(\langle t \rangle)$ would contain the maximal ideal $\mathfrak{m}$, as $R$ is a local ring.  
    \item Assume that $j_0(f_t)=0$ for some relative morsification $f_t$ of $f$. Then we get that $M(\mathfrak{a})=0$ and thus $Jac(f_t)=\mathfrak{a}$. Yet, since $\mathfrak{a}$ is generification of $Jac(f_t)$ over $\mathfrak{p}$, we have that $\sfrac{R}{\mathfrak{a}}$ is Cohen-Macaulay of dimension $2$, and so is $\sfrac{R}{Jac(f_t)}$. Thus, since $t$ is a non zero-disivor of $\sfrac{R}{\mathfrak{a}}$ we have that $\sfrac{R}{Jac(f_t)} \otimes \sfrac{R}{\langle t \rangle} \cong \sfrac{\mathbb{C}\{\underline{x}\}}{Jac(f)}$ is a Cohen-Macaulay ring of dimension $1$. Hence the result follows from Proposition~\ref{prop:trans_disc_LR}. 
\end{enumerate}

\end{proof}

\begin{remark}
\textup{From Lemma~\ref{lemma:j=0_and_disc} we can conclude that there exists  a generification $\mathfrak{a}$ of $Jac(f_t)$ over $\mathfrak{p}$ such that $j_0(\mathfrak{a})=0$ if and only if $\sfrac{R}{Jac(f_t)}$ is a Cohen-Macaulay ring. }
\end{remark}

\begin{lemma}\label{lemma:j_0_finite}
Let  $f_t$ be a relative morsification of $f$ and let $\mathfrak{a}$ be a generification of $Jac(f_t)$ over $\mathfrak{p}$ such that $j_0(\mathfrak{a}) \neq 0$. Then $j_0(\mathfrak{a})$ is finite if and only if for every $m$ we have that $t^m \cdot \mathfrak{a} \not\subset Jac(f_t)$. 
\end{lemma}

\begin{proof}
    Since $j_0(\mathfrak{a}) \neq 0$, then  by Lemma~\ref{lemma:j=0_and_disc}, we have that $M(\mathfrak{a}) \neq 0$, which tells us that $\text{Supp}_R(M(\mathfrak{a})) \neq \emptyset$. Thus by Lemma~\ref{lemma:M(f_t)_CM} we get that $M(\mathfrak{a})$ is a module of dimension $1$ over $R$. Therefore, since $\text{Supp}_{R} (M(\mathfrak{a}) \otimes_R \sfrac{R}{\langle t \rangle }) = \text{Supp}_R(M(\mathfrak{a})) \cap V(\langle t \rangle)$, we have that either the intersection is exactly the unique maximal ideal or $\text{Supp}(M(\mathfrak{a})) \subset V(\langle t \rangle)$. \\
    
    If their intersection is the maximal ideal, that is, $\text{Supp}_{R} (M(\mathfrak{a}) \otimes_R \sfrac{R}{\langle t \rangle }) = \{\mathfrak{m}\}$, then since $M(\mathfrak{a}) \otimes_R \sfrac{R}{\langle t \rangle }$ is a finitely generated module over $\mathbb{C}\{\underline{x}\}$, from Nakayama's Lemma (see Corollary 4.8 in~\cite{eisenbud2013commutative}) we can conclude that $j_0(\mathfrak{a})$ is finite. \\
    
    Otherwise,  $\text{Supp}(M(\mathfrak{a})) \subset V(\langle t \rangle)$ is equivalent to the fact $\sqrt{Ann(M(\mathfrak{a}))} \supset \langle t \rangle$, which itself is equivalent to the fact that there exists some $m$ such that  $t^m \cdot \mathfrak{a} \subset Jac(f_t)$. In this case we have that $\text{Supp}_R (M(\mathfrak{a})) = \text{Supp}_R (M(\mathfrak{a}) \otimes \sfrac{R}{\langle t \rangle})$, which is of dimension $1$, and $j_0(\mathfrak{a})$ can not be finite. 
\end{proof}

\begin{remark}
\textup{The assumption in Lemma~\ref{lemma:j_0_finite} that $j_0(\mathfrak{a}) \neq 0$ is essential. Take for example $f(x,y,z)=x^p+y^p$ and consider the trivial deformation $f_t=f$. Since $f$ is Morse outside $V(\langle x,y\rangle)$ and since $\Delta^\perp(f) =\emptyset$ we have that $f_t$ is a relative morsification of $f$ and that $\mathfrak{a}=Jac(f_t)$ is a generification of itself over $\mathfrak{p}$. Yet we have that $j_0(f_t) =0$, which is finite, and that $t^m \cdot \mathfrak{a} \subset Jac(f_t)$ for every $m$. }
\end{remark}

\begin{lemma}\label{lemma:j_additive}
Let $f \in I^p \setminus I^{p+1}$ with $\Sing(V(f))=V(I)$, let $f_t$ be a relative morsification of $f$, and let $ \mathfrak{a}$ be a generification of $Jac(f_t)$ over $\mathfrak{p}$ such that $j_0(\mathfrak{a})$ is finite and non-zero. Then, following the notations of Notation~\ref{def:all_hail_his_satanic_majesty},  for every $t_0$ small enough we have that $j(\mathfrak{a},p,t_0)$ is finite for every $p \in \pi^{-1}(t_0)$ and  
\begin{equation*}
    j_0(\mathfrak{a}) = \sum_{p \in \pi^{-1}(t_0)} j(\mathfrak{a},p, t_0). 
\end{equation*}
\end{lemma}

\begin{proof} 
Since $j_0(\mathfrak{a})$ is non-zero and finite then from  Lemma~\ref{lemma:j=0_and_disc} and Lemma~\ref{lemma:M(f_t)_CM}, $M(\mathfrak{a})$ is a Cohen-Macaulay module of dimension $1$ over $R$, and the result follows from Proposition~\ref{prop:bigbigbig}. 
\end{proof}

We are now ready to prove Theorem~\ref{thm:main}. 

\begin{proof}[Proof of Theorem~\ref{thm:main}]
First, if $j_0(\mathfrak{a})$ is infinite then the statement is true vacuously, and if $j_0(\mathfrak{a})=0$ then from Lemma~\ref{lemma:j=0_and_disc} we have that $\deg(\Delta^\perp(f)) = 0$, and from Remark~\ref{rem:empty_discr} we can conclude that $\# A_1(f_{t_0})=0$ for every small $t_0$. Hence we can assume that $j_0(\mathfrak{a})$ is finite and non zero. \\

Following Notation~\ref{def:all_hail_his_satanic_majesty}, let $t_0 \in U_1$ and let $ p \in \pi^{-1}(t_0)$. If $p \notin \Crit(f_{t_0})$ then $j(\mathfrak{a},p, t_0)=0$ since $\sfrac{\mathcal{O}_{\mathbb{C}^{n+1}}}{\mathcal{I}_F}$ and $\sfrac{\mathcal{O}_{\mathbb{C}^{n+1}}}{\mathcal{J}}$ are supported inside the critical locus of $F$ with respect to $x_1, \dots, x_n$. If $p$ is a non singular critical point of $f_{t_0}$ then $p$ is a Morse point, since $f_{t_0}$ is Morse outside $V(I)$. As $p \notin V(\mathfrak{p}) \cap U$ and $\sqrt{\mathfrak{a}}=\mathfrak{p}$, then we have that $(\mathcal{I}_F)_{(p,t)} \cong \mathbb{C}\{\underline{x} - \underline{p}\}$.  By applying Morse's Lemma (Theorem 2.46 in \cite{greuel2007introduction}) we have that $j(\mathfrak{a},p, t_0)=1$. So, combined with Lemma~\ref{lemma:j_additive}, we can conclude that  
\begin{equation*}
    j_0(\mathfrak{a}) = \sum_{p \in \Crit(f_{t_0})} j(\mathfrak{a},p, t_0)=\#A_1(f_t) + \sum_{p \in \Sing(f_{t_0})} j(\mathfrak{a},p, t_0). 
\end{equation*}

From Lemma~\ref{lemma:j=0_and_disc}, we get that for every $p \in \Delta^\perp(f_{t_0})$ we have  $j(\mathfrak{a}, p, t_0) \geq 1$. Since $f_t$ is a relative morsification of $f$, then $\Delta^\perp(f_{t_0})$ is reduced, and from Remark~\ref{prop:disc_flat_defor} we can conclude that the size of $\Delta^\perp(f_{t_0})$, as a set, is exactly $\deg(\Delta^\perp(f))$. Thus, we get that 

\begin{equation*}
    j_0(\mathfrak{a}) = \#A_1(f_t) + \sum_{p \in \Sing(f_{t_0})} j(\mathfrak{a},p, t_0) \geq \#A_1(f_t) + \sum_{p \in \Delta^\perp(f_{t_0})} j(\mathfrak{a},p, t_0), 
\end{equation*}

 and since $\sum_{p \in \Delta^\perp(f_{t_0})} j(\mathfrak{a},p, t_0) \geq \deg(\Delta^\perp(f))$, the result follows.\\ 
\end{proof}

\begin{remark}\label{rem:34}
    \textup{From the proof of Theorem~\ref{thm:main} and from part 1 of Remark~\ref{rem:24}, we can conclude that for every relative morsification $f_t$ of $f$, the number of $\# A_1(f_{t_0})$ does not depend on $t_0$ and is finite (as long as $t_0$ is small enough). }
\end{remark}

We now discussion finding an explicit generifications of $Jac(f_t)$ over $\mathfrak{p}$. Since we are interested in $M(\mathfrak{a}) = \sfrac{\mathfrak{a}}{Jac(f_t)}$, one would like to find a "nice" ideal $Jac(f_t) \subset \mathfrak{a}$ which is a generification of $Jac(f_t)$ over $\mathfrak{p}$. The following definitions and  two propositions would give us some motivation for such an ideal. 

\begin{definition}\label{def:for_a _moment}
   The \textbf{generic Jacobian ideal of $f$ (over $I$)} is defined as $Jac_{gen}(f) = \mathbb{C}\{\underline{x}\} \cap Jac(f)_{\langle x_1, \dots, x_{n-1} \rangle}$ and \textbf{the Jacobi number of $f$ (over $I$)} is defined to be $j(f) = \dim_\mathbb{C} (\sfrac{Jac_{gen}(f)}{Jac(f)})$. 
\end{definition}

\begin{remark}
    \textup{Note that }
    \begin{equation*}
        Jac_{gen}(f)=\{ g \in \mathbb{C}\{\underline{x}\} \colon \exists s \notin \langle x_1, \dots, x_{n-1} \rangle\text{ such that } sg \in Jac(f)\}. 
    \end{equation*}
    \textup{In addition, we can compare the definition of $Jac_{gen}(f)$ with the definition of $I_f$, as presented in Section 5 of~\cite{pellikaan1988finite}. }
\end{remark}

\begin{proposition}\label{prop:Jac_gen_is_generification}
 $\sfrac{\mathbb{C}\{\underline{x}\}}{ Jac_{gen}(f)}$ is a Cohen-Macaulay ring of dimension $1$. In addition, $Jac_{gen}(f)$ is the smallest such ideal which contains $Jac(f)$. 
\end{proposition}

\begin{proof}
First, $x_n$ is a non zero-divisor of $\sfrac{\mathbb{C}\{\underline{x}\}}{ Jac_{gen}(f)}$ since if $x_nh \in  Jac_{gen}(f)$, then there exists some $s \notin \langle x_1, \dots, x_{n-1} \rangle$ such that $sx_n h \in Jac(f)$, but since $s x_n \notin \langle x_1, \dots, x_{n-1} \rangle$, we get that $h \in Jac_{gen}(f)$. Thus the depth of $\sfrac{\mathbb{C}\{\underline{x}\}}{ Jac_{gen}(f)}$ is at least $1$. Yet we have that $Jac(f) \subset Jac_{gen}(f)$, and thus $V(Jac_{gen}(f)) \subset V(Jac(f))$. So we have that $\dim(\sfrac{\mathbb{C}\{\underline{x}\}}{Jac_{gen}(f)}) \leq \dim(\sfrac{\mathbb{C}\{\underline{x}\}}{Jac(f)})=1$, and 
 we can conclude that $\sfrac{\mathbb{C}\{\underline{x}\}}{ Jac_{gen}(f)}$ is a Cohen-Macaulay ring of dimension $1$.\\

Now, let $Jac(f) \subset \mathfrak{a}$ be an ideal such that $\sfrac{\mathbb{C}\{\underline{x}\}}{\mathfrak{a}}$ is a Cohen Macaulay ring of dimension $1$, and let $h \in Jac_{gen}(f)$. Then there exists some $s \notin \langle x_1, \dots, x_{n-1} \rangle$ such that $sh \in Jac(f)$. But since $Jac(f) \subset \mathfrak{a}$ we get that $sh \in \mathfrak{a}$. Yet, since $\sfrac{\mathbb{C}\{\underline{x}\}}{\mathfrak{a}}$ is a Cohen Macaulay ring of dimension $1$ we have that $s$ must be a non zero-divisor of $\sfrac{\mathbb{C}\{\underline{x}\}}{\mathfrak{a}}$, and thus  $h \in \mathfrak{a}$.  
\end{proof}

\begin{proposition}
   $j(f)$ is finite, and equals to zero if and only if $\sfrac{\mathbb{C}\{\underline{x}\}}{Jac(f)}$ is Cohen-Macaulay over $\mathbb{C}\{\underline{x}\}$. 
\end{proposition}

\begin{proof}
Since $\sfrac{Jac_{gen}(f)}{Jac(f)}$ is finitely generated over 
$\mathbb{C}\{\underline{x}\}$, it is enough to prove that its support is only $\mathfrak{m}$. First, since $V(Jac(f))=V(Jac_{gen}(f))=V(I)$ then for every $\mathfrak{q} \notin V(I)$ we have that $(Jac_{gen}(f))_\mathfrak{q} = (Jac(f))_\mathfrak{q}=0$. Now, since as a subset of $\text{Spec}(\mathbb
{C}\{\underline{x}\})$ we have that $V(I) =\{I, \mathfrak{m}\}$, and since $(Jac(f))_I = (Jac_{gen}(f))_I$, we indeed get that the support is exactly the maximal ideal. The second part follows from Proposition~\ref{prop:Jac_gen_is_generification}. 
\end{proof}

Therefore, a natural idea would be to look at $j_0(Jac_{gen}(f_t))$ where we set $Jac_{gen}(f_t)=R \cap Jac(f_t)_\mathfrak{p}$ since (as with $Jac_{gen}(f)$) we have that  $(Jac_{gen}(f_t))_\mathfrak{p}=Jac(f_t)_\mathfrak{p}$, $\dim(\sfrac{R}{Jac_{gen}(f_t)}) \leq 2$, and $t$ is a non zero-divisor of $\sfrac{R}{Jac_{gen}(f_t)}$. But, as we see in the next example, the quotient ring $\sfrac{R}{Jac_{gen}(f_t)}$ need not be a Cohen-Macaulay ring for any $f_t$. 

\begin{example}
\textup{Let $f_t(x,y,z)=tx^p +x^{p+1}y+zy^{p-2}x^3 +y^p$. We show that the depth of $\sfrac{R}{Jac_{gen}(f_t)}$ is $1$ by showing that $z$ is a zero divisor of $\sfrac{R}{Jac_{gen}(f_t)+\langle t \rangle}$, since $t$ is a non zero-divisor of $\sfrac{R}{Jac_{gen}(f_t)}$. First, we have that}
\begin{equation*}
    Jac(f_t)= \langle ptx^{p-1}+(p+1)x^py +3zx^2y^{p-2}, x^{p+1} + (p-2)zy^{p-2}x^3 +py^{p-2}, y^{p-2}x^3\rangle. 
\end{equation*}

\textup{Thus we get that $ptx^p+(p+1)x^{p+1} y \in Jac(f_t)$. But since $ptx^p+(p+1)x^{p+1}y=x^{p}(pt + (p+1)xy)$ and $(pt + (p+1)xy) \notin Jac_{gen}(f_t)$, we get that $x^p \in Jac_{gen}(f_t)$. Thus we can conclude that $ptx^{p-1} + zx^2 y^{p-2} \in Jac_{gen}(f_t)$ and thus $zx^2 y^{p-2} \in Jac_{gen}(f_t)+\langle t \rangle$. Yet we have that $x^2 y^{p-2} \notin Jac_{gen}(f_t)+\langle t \rangle$.}
\end{example}

Yet, if we assume that $j(f) \neq 0$ then this is true, as we see in the following lemmata. 

\begin{lemma} \label{lemma:crazy_man_prime_avoidance}
    Let $\mathfrak{a} \subset R$ be any ideal and let $\mathfrak{a} = \cap_{i=1}^k \mathfrak{a}_i$ be the prime decomposition of $\mathfrak{a}$, where $\mathfrak{a}_1, \dots, \mathfrak{a}_l \subset \mathfrak{p}$ and $\mathfrak{a}_{l+1}, \dots, \mathfrak{a}_k \not\subset \mathfrak{p}$. Denote $\sqrt{\mathfrak{a}_i}=\mathfrak{q}_i$ for every $i$ and $\mathfrak{b}= \cap_{i=1}^l \mathfrak{a}_i$. Then $\cap_{i=l+1}^k \mathfrak{p}_i \not\subset \mathfrak{p} \cup \langle t \rangle$ if and only if for every $r \in \mathfrak{b}$ there exists some $s \notin \mathfrak{p} \cup\langle t \rangle$ such that $sr \in \mathfrak{a}$. 
\end{lemma}

\begin{proof}
First, assume that  $\cap_{i=l+1}^k \mathfrak{p}_i \not\subset \mathfrak{p} \cup \langle t \rangle$. Then by the Prime Avoidance Lemma (see Section 3.2 in~\cite{eisenbud2013commutative}) there exists some $s \in (\cap_{i=l+1}^k \mathfrak{p}_i) \setminus (\mathfrak{p} \cup \langle t \rangle)$. Therefore there exists some $N$ such that $s^N \in \cap_{i=l+1}^k \mathfrak{a}_i$ and $s^N \notin \mathfrak{p} \cup \langle t \rangle$. So for every $r \in \mathfrak{b}$ we have that $r \in \cap_{i=1}^l \mathfrak{a}_i$ and so $s^N r \in \cap_{i=1}^k \mathfrak{a}_i =\mathfrak{a}$. \\

Second, assume that for every $r \in \mathfrak{b}$ there exists some $s \notin \mathfrak{p} \cup\langle t \rangle$ such that $sr \in \mathfrak{a}$. Let $c \in \mathfrak{b} \setminus \mathfrak{a}$, and let $s \notin \mathfrak{p} \cup \langle t \rangle$ such that $sc \in \mathfrak{a}$. Then for every $i$ we have that $sc \in \mathfrak{a}$. But since for every $i$, $\mathfrak{a}_i$ is primary with $r \notin \mathfrak{a} \subset \mathfrak{a}_i$, we have that $s \in \sqrt{\mathfrak{a}_i}=\mathfrak{p}_i$. Thus we can conclude that $s \in \cap_{i=l+1}^k \mathfrak{p}_i$ but $s \notin \mathfrak{p} \cup \langle t \rangle$. 

\end{proof}

\begin{lemma}\label{lemma:p_and_Jac(f_t)_and_t}
    Assume that $j(f) \neq 0$. Then for every $r \in Jac_{gen}(f_t)$ there exists some $s \notin \mathfrak{p} \cup \langle t \rangle$ such that $sr \in Jac(f_t)$. 
\end{lemma}

\begin{proof}
    Let $Jac(f_t) = \cap_{i=1}^k \mathfrak{a}_i$ be the prime decomposition of $Jac(f_t)$. Now, since $Jac(f_t)_\mathfrak{p}= Jac_{gen}(f_t)_\mathfrak{p}$ we can assume that $\sqrt{\mathfrak{a}_1}=\mathfrak{p}$. Since $j(f) \neq 0$, then $Jac_{gen}(f_t) \neq Jac(f_t)$, and so from the proof of Lemma~\ref{lemma:M(f_t)_CM} the only $2-$dimensional irreducible component of $Jac(f_t)$ is $V(\mathfrak{p})$. Therefore we have that $\mathfrak{a}_i \not\subset \mathfrak{p}$ for every $i \neq 1$. Now, since $V(\mathfrak{a}_i)$ is a curve for every $i>1$ and $\cup_{i>1} V(\mathfrak{a}_i) = V(\cap_{i>1} \mathfrak{a}_i)$, we get that $\cap_{i>1} \mathfrak{a}_i \not\subset \langle t\rangle$ in addition to $\cap_{i>1} \mathfrak{a}_i \not\subset \mathfrak{p}$, and the result follows from Lemma~\ref{lemma:crazy_man_prime_avoidance}. 
\end{proof}

\begin{lemma}\label{lemma:s_q_and_stuff}
    Denote $\mathfrak{q} =\langle x_1, \dots, x_{n-1}, t \rangle$ and for every $\mathfrak{a} \subset R$ define $s_\mathfrak{q} (\mathfrak{a})=R \cap \mathfrak{a}_\mathfrak{q}$. Then $s_\mathfrak{q}(Jac(f_t) +\langle t\rangle) = s_\mathfrak{q}(Jac(f_t)) +\langle t\rangle$. 
\end{lemma}

\begin{proof}
    First, note that $s_\mathfrak{q}(Jac(f_t) +\langle t\rangle) = s_\mathfrak{q}(Jac(f_t)) +s_\mathfrak{q}(\langle t\rangle)$ (as $Jac(f_t) , \langle t \rangle \subset \mathfrak{q}$ but they are not contained in each other) and $\langle t \rangle \subset s_\mathfrak{q}(\langle t \rangle)$. Now, let $r \in s_\mathfrak{q}(\langle t \rangle)$. Then there exists some $s \notin \mathfrak{q}$ such that $rs \in \langle t \rangle$. But since $\langle t \rangle \subset \mathfrak{q}$ is a prime ideal, then $r \in \langle t \rangle$. 
\end{proof}

\begin{proposition}\label{prop:Jac_f_t_and_Jac_gen}
    Assume that $j(f) \neq 0$. Then $Jac_{gen}(f_t) +\langle t \rangle
    = Jac_{gen}(f) \cdot R +\langle t \rangle$. 
\end{proposition}

\begin{proof}
    First, let $h \in Jac_{gen}(f_t) +\langle t \rangle$. Then there exists some $b \in R$ such that $h-bt \in Jac_{gen}(f_t)$, and we can assume that $h \notin \langle t\rangle$. By Lemma~\ref{lemma:p_and_Jac(f_t)_and_t}, there exists some $s \notin \mathfrak{p} \cup \langle t\rangle$ such that $s(h-bt) \in Jac(f_t)$. Now, we can write $s=s_0+t \tilde{s}$ where $0 \neq s_0 \in \mathbb{C}\{\underline{x}\}$, and we have that $s_0h + t(\tilde{s} h -bh) =\sum_{i=1}^n a_i \partial_i(f_t)$ for some $a_1, \dots, a_n \in R$. Now recall that $f_t=f+tg$ for $f,g \in \mathbb{C}\{\underline{x}\}$, so if we write $a_i=c_i +td_i$ for some $c_1, \dots, c_n \in \mathbb{C}\{\underline{x}\}$ and some $d_1, \dots, d_n \in R$, then we have that $s_0 h + t(\tilde{s}h -bh + \sum_{i=1}^n a_i\partial_i(g)  +\sum_{i=1}^n d_i \partial_i(f) ) =\sum_{i=1}^n c_i \partial_i(f) \in Jac(f) \subset \mathbb{C}\{\underline{x}\}$. If $c_i=0$ for every $i$ then $s_0h=0$, and thus $h =0$. Otherwise, by comparing coefficients we have that $s_0h \in Jac(f)$, which tells us that $h\in Jac_{gen}(f)$.  \\

    Second, let $h \in Jac_{gen}(f) \cdot R +\langle t \rangle$. Then there exists some $b \in R$, some $r_1, \dots, r_n \in R$, and some $c_1, \dots, c_n \in Jac_{gen}(f)$ such that $b-ht =\sum_i r_i c_i$, and we can assume that $h \in \mathbb{C}\{\underline{x}\}$. Thus there exists some $s \notin I \subset \mathbb{C}\{\underline{x}\}$ such that $sr_i \in Jac(f)$ for every $i$. Therefore there exists some $d_1, \dots, d_n \in R$ such that $s(h-bt) = \sum_{i=1}^n d_i \partial_i(f)$, and so $sh+t(\sum_{i=1}^n c_i \partial_i(g) -sb)=\sum_{i=1}^n d_i \partial_i(f_t) \in Jac(f_t)$. Hence $sh \in Jac(f_t) +\langle t \rangle $, which gives us that $ h \in s_\mathfrak{q}(Jac(f_t) +\langle t\rangle)$. But, from Lemma~\ref{lemma:s_q_and_stuff} we have that $s_\mathfrak{q}(Jac(f_t) +\langle t\rangle) = s_\mathfrak{q}(Jac(f_t)) +\langle t\rangle \subset Jac_{gen}(f_t) +\langle t \rangle $, and the result follows. 
\end{proof}

\begin{proposition}\label{prop:Jac_gen_f_t_and_f}
If $j(f) \neq 0$ then  $\sfrac{R}{Jac_{gen}(f_t)}$ is a Cohen-Macaulay module over $R$ of dimension $2$ and $j_0(Jac_{gen}(f_t)) = j(f)$. 
\end{proposition}

\begin{proof}
First, recall that from Proposition~\ref{prop:Jac_gen_is_generification} we have that $\sfrac{\mathbb{C}\{\underline{x}\}}{Jac_{gen}(f)}$ is a Cohen-Macaulay module over $\mathbb{C}\{\underline{x}\}$ of dimension $1$. Therefore, it is enough to prove that $t$ is a non zero-divisor of $\sfrac{R}{Jac_{gen}(f_t)}$ and that $\sfrac{R}{Jac_{gen}(f_t)} \otimes_R \sfrac{R}{\langle t \rangle } \cong \sfrac{\mathbb{C}\{\underline{x}\}}{Jac_{gen}(f)}$. First, assume that $th \in Jac_{gen}(f_t)$. Then there exists some $s \notin \mathfrak{p}$ such that $sth \in Jac(f_t)$, but since $t \notin \mathfrak{p}$ as well, then we have that $h \in Jac_{gen}(f_t)$. Therefore $t$ is a non zero-divisor of $\sfrac{R}{Jac_{gen}(f_t)}$. Second, from Proposition~\ref{prop:Jac_f_t_and_Jac_gen} we have that $Jac_{gen}(f_t) +\langle t \rangle = Jac_{gen}(f) \cdot R +\langle t \rangle$, and so $\sfrac{R}{Jac_{gen}(f_t)} \otimes_R \sfrac{R}{\langle t \rangle } \cong \sfrac{R}{Jac_{gen}(f_t) +\langle t \rangle} \cong \sfrac{R}{Jac_{gen}(f)\cdot R +\langle t \rangle} \cong \sfrac{\mathbb{C}\{\underline{x}\}}{Jac_{gen}(f)}$.\\

Now in order to prove that $j_0(Jac_{gen}(f_t)) = j(f)$ it is enough to prove that $\sfrac{Jac_{gen}(f)}{Jac(f)} \cong \sfrac{Jac_{gen}(f_t)}{Jac(f_t)} \otimes_R \sfrac{R}{\langle t \rangle }$. Since $Jac_{gen}(f_t) +\langle t \rangle = Jac_{gen}(f) \cdot R +\langle t \rangle$, then every element of $h \in Jac(f) \subset \mathbb{C}\{\underline{x}\}$ can be written as $h = h_0+tb$ for $h_0 \in Jac_{gen}(f_t)$ and $b \in R$. Therefore we can define a map $\Phi \colon Jac_{gen}(f) \to \sfrac{Jac_{gen}(f_t)}{Jac(f_t)} \otimes_R \sfrac{R}{\langle t \rangle }$ by $\Phi(h) = (h_0 + Jac(f_t)) \otimes (1+\langle t \rangle)$. Then by direct computation we can conclude that $\Phi$ is a well defined map. surjective, and $\ker(\Phi)=Jac(f)$. Thus, the result follows from the first isomorphism theorem.  

\end{proof}

\begin{theorem}\label{thm:main_Jac_gen}
For every relative morsification $f_t$ of $f$ and for every small enough $t_0$ we have that 
    \begin{equation*}
            j(f) \geq \# A_1(f_{t_0}) +\deg(\Delta^\perp(f)).
    \end{equation*}
\end{theorem}

\begin{proof}
    If $\Delta^\perp(f) =\emptyset$ then from Remark~\ref{rem:empty_discr} we have that $\#A_1(f_t)=0$ and the result is true vacuously. Otherwise, From Proposition~\ref{prop:Jac_gen_f_t_and_f} we have that $Jac_{gen}(f_t)$ is a generification of $Jac(f_t)$ and the result follows from Theorem~\ref{thm:main}. 
\end{proof}

\begin{example}\label{example:uuuuuu}

\textup{One can calculate the following examples:}

\begin{center}
\hspace*{-1.3cm}
    \begin{tabular}{ | l | l | l | l |}
    \hline
    $f$ & $j(f)$   & $\deg (\Delta^\perp(f))$ \\ \hline \hline
    $x^p+y^pz$ & $p-1$ & $p-1$  \\ \hline
    $x^p+y^pz^q+y^{p+1}$ & $(2q-1)(p-1)$ &  $q(p-1)$ \\ \hline
    $x^pz^{q_1}+y^pz^{q_2}+y^{p+1}+x^{p+1}$ &  $(p-2)(2(q_1+q_2)-1)+2(q_1+q_2)$ & $(q_1+q_2)(p-1)$ \\ \hline 
    $\prod_{i=1}^n (x^{p_i}+y^{p_i}z)$ &  $\sum_{i=1}^n (p_i-1)$  & $\sum_{i=1}^n (p_i-1)$ \\ \hline
    $\prod_{i=1}^n(x^{p_i}+y^{p_i}z^{q_i}+y^{p_i+1})$ & $\sum_{i=1}^n (2q_i-1)(p_i-1)$ & $\sum_{i=1}^n q_i(p_i-1)$ \\ \hline 
    $\sum_{i=1}^{n-2} (x_i^p) +x_{n-1}^px_n^q +x_{n-1}^{p+1}$ &  $(2q-1)(p-1)^{n-2}$ & $q(p-1)^{n-2}$ \\ \hline
    $\sum_{i=1}^{n-2} (x_i^p x_n^{q_i}+x_i^{p+1})$ & $2^{n-2}(\sum_{i=1}^{n-2} q_i)(p-1)^{n-2}+ (2^{n-2}-1)p$ & $(\sum_{i=1}^n q_i)(p-1)^{n-2}$ \\ \hline 
    
    \end{tabular}
    
    \vspace{5mm}

\end{center}
\end{example}

Comparing the table above with Example~\ref{example:table_mors} we can see that Theorem~\ref{thm:main_Jac_gen} is true in all of these cases (and we even have an equality). Yet, the following example shows that the inequality in Theorem~\ref{thm:main_Jac_gen} need not be an equality.  

\begin{example}\label{example:inequality}
\textup{Let $f(x,y,z)=2k(x^{3k} +y^{3k}) -3k z^2 x^{2k} y^{2k}$ for $k \geq 2$. As we saw in Remark~\ref{rem:counterex}, $\Delta^\perp(f) = \emptyset$ but $Jac(f) \neq Jac_{gen}(f)$, since $\sfrac{\mathbb{C}\{\underline{x}\}}{Jac(f)}$ is not a Cohen-Macaulay ring, in contrast to  $\sfrac{\mathbb{C}\{\underline{x}\}}{Jac_{gen}(f)}$, as we have seen in Proposition~\ref{prop:Jac_gen_is_generification}. Thus we get that $j(f) > 0$ but $\deg(\Delta^\perp(f))=0=\# A_1(f_{t_0})$ for every small $t_0$.}
\end{example}

\section{Milnor Number of Yomdin-Type Singularities}\label{sec:milnor}
In this section we discuss how we can apply the results from Section~\ref{section:SPJ} to generalize some of the results in~\cite{pellikaan1989series}, which in addition gives us a bound for $\deg(\Delta^\perp(f))$ by an algebraic invariant. We start by recalling a few results from Section 3 of~\cite{pellikaan1989series}, which we use throughout this section.

\begin{proposition}\label{prop:iso_sing_f+x_n^k}
    Let $f \in I^p \setminus I^{p+1}$ with $\Sing(V(f))=V(I)$. Denote $J_{n-1}(f)=\langle \partial_1(f), \dots, \partial_{n-1}(f) \rangle \subset \mathbb{C}\{\underline{x}\}$. Then 
    \begin{enumerate}
        \item $\sfrac{R}{{J_{n-1}(f)}}$ is a Cohen-Macaulay ring of dimension $1$.
        \item For large enough $ k\in \mathbb{N}$ we have that $f+x_n^k$ has an isolated singularity at the origin and its Milnor number is
\begin{equation*}
    \mu(f+x_n^k)=j(f)+(k-1)(p-1)^{n-1}+ \dim_\mathbb{C}(\sfrac{Jac(f)}{Jac_{gen}(f) \cap Jac(f+x_n^k)}). 
\end{equation*}
    \item For large enough $k \in \mathbb{N}$ we have that  $Jac_{gen}(f) \cap Jac(f+x_n^k) = J_{n-1}(f) + \langle \partial_n(f) \rangle \cdot Jac_{gen}(f)$. 
    \end{enumerate}
\end{proposition}

\begin{remark}\begin{enumerate}
    \item     \textup{
 Since $f \in I^p \setminus I^{p+1}$, the generic multiplicity of $Jac_{gen}(f)$ is $(p-1)$, and thus, since set theoretically, $V(Jac_{gen}(f))=V(x_1, \dots, x_{n-1})$,  we get that $\dim_\mathbb{C}(\sfrac{\mathbb{C}\{\underline{x}\}}{Jac_{gen}(f)+\langle x_n \rangle})=\text{mult}(Jac_{gen}(f), \langle x_n \rangle)=(p-1)^{n-1}$. }

    \item \textup{Recall that the Milnor number of some $g \in \mathbb{C}\{\underline{x}\}$ is defined to be $\mu(g)=\dim_\mathbb{C}(\sfrac{\mathbb{C}\{\underline{x}\}}{Jac(f)})$. The milnor number plays a crucial role in the theory of isolated singularities, as discussed in length in Section 2 of Chapter 1 of`\cite{greuel2007introduction}. } 
\end{enumerate}

\end{remark}

We are now interested in computing and bounding the complex dimension of $\sfrac{Jac(f)}{Jac_{gen}(f) \cap Jac(f+x_n^k)}$ (for large enough $k$), which we denote by ${\delta}(f)$. Note that $\delta(f)$ is finite since $\mu(f +x_n^k)$ is finite for large enough $k$ and $\delta(f) \leq \mu(f +x_n^k)$, from Proposition~\ref{prop:iso_sing_f+x_n^k}. In addition, note that $\delta(f)$ does not depend on $k$, since by Proposition~\ref{prop:iso_sing_f+x_n^k} we have that $\sfrac{Jac(f)}{Jac_{gen}(f) \cap Jac(f+x_n^k)} \cong \sfrac{Jac(f)}{J_{n-1}(f) + \langle \partial_n(f) \rangle \cdot Jac_{gen}(f)}$ (for large enough $k$).\\

\begin{proposition}\label{prop:delta_colon}
    $ \sfrac{\mathbb{C}\{\underline{x}\}}{(Jac(f+x_n^k) \colon \langle \partial_n(f) \rangle )} \cong  \sfrac{Jac(f)}{Jac_{gen}(f) \cap Jac(f+x_n^k)}$
\end{proposition}

\begin{proof}
Note that $\sfrac{Jac(f)}{Jac_{gen}(f) \cap Jac(f+x_n^k)}$ is generated by $\partial_n(f)$, since $\partial_i(f) \in Jac(f+x_n^k)$ for every $i<n$. Thus we have a surjective map 
\begin{equation*}
    \varphi \colon \mathbb{C}\{\underline{x}\} \to \sfrac{Jac(f)}{Jac_{gen}(f) \cap Jac(f+x_n^k)}
\end{equation*}

defined by $\varphi(r)=r \cdot \partial_n(f)$ with $\ker (\varphi)=\text{Ann}( \sfrac{Jac(f)}{Jac_{gen}(f) \cap Jac(f+x_n^k)})$. Since $\partial_i(f) \in Jac(f+x_n^k)$ for every $i<n$, we get that $\ker (\varphi)=(Jac(f+x_n^k) \colon \langle \partial_n(f) \rangle )$. Hence, from the first isomorphism theorem we get that 
\begin{equation*}
    \sfrac{Jac(f)}{Jac_{gen}(f) \cap Jac(f+x_n^k)} \cong \sfrac{\mathbb{C}\{\underline{x}\}}{(Jac(f+x_n^k) \colon \langle \partial_n(f) \rangle )},
\end{equation*}

 as desired. 
\end{proof}

\begin{lemma} \label{lemma:CM_delta}
    Let $f_t$ be a relative morsification of $f$. Then for a large enough $k$ we have that $\sfrac{R}{(Jac(f_t+x_n^k) \colon \langle \partial_n(f_t) \rangle )}$ is a Cohen-Macaulay module over $R$ of dimension $1$ and 
    \begin{equation*}
        \sfrac{R}{(Jac(f_t+x_n^k) \colon \langle \partial_n(f_t) \rangle )} \otimes \sfrac{R}{\langle t \rangle} \cong \sfrac{\mathbb{C}\{\underline{x}\}}{(Jac(f+x_n^k) \colon \langle \partial_n(f) \rangle )},
    \end{equation*} 
    where, as in Notation~\ref{def:all_hail_his_satanic_majesty}, we denote $R=\mathbb{C}\{\underline{x},t\}$. 
\end{lemma}

\begin{proof}
    First, from Proposition~\ref{prop:iso_sing_f+x_n^k} we have that $\sfrac{\mathbb{C}\{\underline{x}\}}{Jac(f+x_n^k)}$ is zero dimensional for large enough $k$. Thus $\sfrac{R}{Jac(f_t+x_n^k)}$ is one dimensional in this case, and since we know that $Jac(f_t+x_n^k) \subset (Jac(f_t+x_n^k) \colon \langle \partial_n(f_t) \rangle )$, we can conclude that that $\dim(\sfrac{R}{(Jac(f_t+x_n^k) \colon \langle \partial_n(f_t) \rangle )}) \leq 1$. Therefore it is enough to show that $t$ is a non zero-divisor of $\sfrac{R}{(Jac(f_t+x_n^k) \colon \langle \partial_n(f_t) \rangle )}$. Assume that $tr \in (Jac(f_t+x_n^k) \colon \langle \partial_n(f_t) \rangle )$, and then $tr \partial_n(f_t) \in Jac(f_t+x_n^k)$. But  $\sfrac{\mathbb{C}\{\underline{x}\}}{Jac(f+x_n^k)}$ is zero dimensional and $Jac(f+x_n^k)$ is generated by $n$ elements in $R$, therefore $\sfrac{R}{Jac(f_t+x_n^k)}$ is a Cohen-Macaulay module and in particular, $t$ is a non zero-divisor of $\sfrac{R}{Jac(f_t+x_n^k)}$. Hence since $tr \partial_n(f_t) \in Jac(f_t+x_n^k)$ we get that $r \partial_n(f_t) \in Jac(f_t+x_n^k)$, which gives us that $r \in (Jac(f_t+x_n^k) \colon \langle \partial_n(f_t) \rangle )$. Thus $t$ is a non zero-divisor and we can conclude that $\sfrac{R}{(Jac(f_t+x_n^k) \colon \langle \partial_n(f_t) \rangle )}$ is a Cohen-Macaulay module of dimension $1$. \\

    Second, in order to finish the proof it is enough to show that that
    \begin{equation*}
        (Jac(f_t+x_n^k) \colon \langle \partial_n(f_t) \rangle ) + \langle t \rangle = (Jac(f+x_n^k) \colon \langle \partial_n(f) \rangle ) \cdot R +\langle t \rangle.
    \end{equation*}
      Let $r \in (Jac(f_t+x_n^k) \colon \langle \partial_n(f_t) \rangle ) + \langle t \rangle$. Then there exists some $b \in R$ and $r_0 \in \mathbb{C}\{\underline{x}\}$ such that $r=r_0-bt \in (Jac(f_t+x_n^k) \colon \langle \partial_n(f_t) \rangle )$.
     Then $\partial_n(f_t) \cdot (r_0-bt) \in Jac(f_t+x_n^k)$, and so there exists some $c_1, \dots, c_n \in R$ such that $\partial_n(f_t) \cdot (r_0-bt) =\sum_{i=1}^n c_i \partial_i(f_t) + c_n kx_n^{k-1}$. Therefore, if we write $c_i = a_i +\tilde{c}_i t$ where $a_1, \dots, a_n \in \mathbb{C}\{\underline{x}\}$ and $\tilde{c}_1, \dots, \tilde{c}_n \in R$, then there exists some $\tilde{b} \in R$ such that $\partial_n(f) \cdot r_0-\tilde{b}t =\sum_{i=1}^n a_i \partial_i(f) + c_n kx_n^{k-1}$. By comparing coefficients we have that $\partial_n(f) r_0 \in Jac(f+x_n^k)$. So $r_0 \in (Jac(f+x_n^k) \colon \langle \partial_n(f) \rangle )$ and we get that $r \in (Jac(f+x_n^k) \colon \langle \partial_n(f) \rangle) \cdot R +\langle t \rangle$. The other inclusion follows from stability of colon ideal under such a base change (see similarity with Exercise 15.41 in~\cite{eisenbud2013commutative}). 
\end{proof}

\begin{lemma}\label{lemma:delta_zero}
    $\delta(f) =0$ if and only if $\partial_n(f) \in J_{n-1}(f)$. 
\end{lemma}

\begin{proof}
    From Proposition~\ref{prop:delta_colon}, if $\delta(f) =0$ then $(Jac(f+x_n^k) \colon \langle \partial_n(f) \rangle)=\mathbb{C}\{\underline{x}\}$, which is equivalent to $\partial_n(f) \in Jac(f+x_n^k)$. Therefore there exists some $c_1, \dots, c_n \in \mathbb{C}\{\underline{x}\}$ such that $\partial_n(f) = \sum_{i=1}^n c_i \partial_i (f) +c_n k x_n^{k-1}$. Therefore $c_n x_n^k \in Jac(f)$, and since $\sqrt{Jac(f)} =I$ we have that $c_n  \in I$. Yet, we have that $(1-c_n) \partial_n(f) = \sum_{i=1}^{n-1} c_i \partial_i(f) + c_n k x_n^{k-1}$, and since $1-c_n$ is invertible we can conclude that $\partial_n(f) \in J_{n-1}(f) +\langle x_n^{k-1}\rangle$. Hence $Jac_{gen}(f+x_n^{k}) = Jac(f) +\langle x_n^{k-1} \rangle = J_{n-1}(f) +\langle x_n^{k-1} \rangle$, and in particular $\partial_n(f) \in Jac(f+x_n^k)$. 
    By applying Proposition~\ref{prop:iso_sing_f+x_n^k} we can conclude that $\partial_n(f) \in   J_{n-1}(f) + \langle \partial_n(f) \rangle \cdot Jac_{gen}(f)$. Therefore there exists some $g \in Jac_{gen}(f)$ such that $\partial_n(f) + g\cdot \partial_n(f) \in J_{n-1}(f)$, but since $1+g$ is invertible we can conclude that $\partial_n(f) \in J_{n-1}(f)$ as desired. On the other hand, if $\partial_n(f) \in J_{n-1}(f)$ then from Proposition~\ref{prop:iso_sing_f+x_n^k} we get that $Jac_{gen}(f) \cap Jac(f+x_n^k) = J_{n-1}(f) = Jac(f)$, and so $\delta(f)=0$. 
\end{proof}

\begin{remark}\label{remark:delta_zero_trans}
    \textup{From Lemma~\ref{lemma:delta_zero} we can conclude that if $\delta(f)=0$ then  $\Delta^\perp(f)=\emptyset$. This is true since $\partial_n(f) \in J_{n-1}(f)$ is true if and only if up to a coordinate change, $f$ is a function of $x_1, \dots, x_{n-1}$ (for more details see Proposition 1.11 in~\cite{pellikaan1989series} and Section 9 of \cite{pellikaanhypersurface}), and we can combine this result with Remark~\ref{rem:top_trans_disc}.  }
\end{remark}

\begin{theorem}\label{thm:delta_trans_bound}
$\delta(f) \geq \deg(\Delta^\perp(f))$. 
\end{theorem}

\begin{proof}
Let $f_t$ be a relative morsification of $f$, which exists from Theorem~\ref{prop:morsif}. Then by Proposition~\ref{lemma:CM_delta} we have that $\sfrac{R}{(Jac(f_t+x_n^k) \colon \langle \partial_n(f_t) \rangle )}$ is a Cohen-Macaulay module of dimension $1$ for a large enough $k$, and from Remark~\ref{remark:delta_zero_trans} we have that if $\Delta^\perp(f) \neq \emptyset$ then $\delta(f) >0$. Observe that in addition, if $f$ is either smooth or a Morse function, then $\delta(f)=0$. Therefore, if we look at the corresponding sheaf, we get that the support of the corresponding fiber is finite, is contained in $V(I)$, and contains the corresponding transversal discriminant. Therefore the result follows by applying the same technique based upon Proposition~\ref{prop:bigbigbig} which we used in the proof of Theorem~\ref{thm:main}. 
\end{proof}

\begin{corollary}\label{cor:mu_rel_mors}
     Assume that $n\neq 4$. Then for any relative morsification $f_t$ of $f$ we have that $\mu(f+x_n^k) \geq \#A_1(f_{{t_0}}) + (k-1)(p-1)^{n-1} + 2\deg(\Delta^\perp(f))$ for every large enough $k$ and for every small enough $t_0$. 
\end{corollary}

\begin{proof}
    Follows from Theorem~\ref{thm:main_Jac_gen} and Theorem~\ref{thm:delta_trans_bound}. 
\end{proof}

As with Theorem~\ref{thm:main_Jac_gen}, the inequality in Theorem~\ref{thm:delta_trans_bound} need not be an equality, as we see in the following example. 

\begin{example}\label{example:beansbeans}
    \textup{As in Example~\ref{example:inequality}, we have that for $f(x,y,z)=2k(x^{3k} +y^{3k}) -3k z^2 x^{2k} y^{2k}$ with $k \geq 2$, $\Delta^\perp(f) = \emptyset$ but $\partial_n(f) \notin Jac(f+x_n^k)$, and so from Proposition~\ref{prop:delta_colon} we get $\delta(f) >0$. }
\end{example}

We end this section with a proposition and a corollary which examine the relation between a relative morsification $f_t$ of $f$ and the deformation $f_t+x_n^k$ of $f+x_n^k$ (for a large enough $k$). 

\begin{proposition}\label{prop:rel_mors_mu_A_1}
Let $f_t$ be a relative morsification of $f$. Then for large enough $k \in \mathbb{N}$ and for every small enough $t_0$ we have that $\mu(f+x_n^k) = \# A_1(f_{t_0}) + \mu(f_{t_0}+x_n^k)$. 
\end{proposition}

\begin{proof}
By Proposition~\ref{prop:iso_sing_f+x_n^k}, $f+x_n^k$ is an isolated singularity for a large enough $k$, and so $f_{t_0}+x_n^k$ is an isolated singualrity for every small enough $t_0$ as well. Now, for every $A_1$ point $z_0$ of $f_{t_0}$ we have that $\Hess(f_{t_0})(z_0) \in GL_n(\mathbb{C})$, and we have that $\Hess(f_{t_0}+x_n^k)(z_0)=\Hess(f_{t_0})(z_0)+\Hess(x_n^k)(z_0)$. Yet, $GL_n(\mathbb{C}) \subset \mathbb{C}^{n \times n}$ is an open set (with respect to the Euclidean topology), and $f_{t_0}$ has only a finite number of $A_1$ critical points (as mentioned in Remark~\ref{rem:34}). Therefore there exists some $K$ such that for every $k \geq K$ we have that $\Hess(f+x_n^k)(z) \in GL_n(\mathbb{C})$ for every $A_1$ point $z$ of $f_{t_0}$. Hence we can conclude that every $A_1$ point of $f_{t_0}$ is an $A_1$ point of $f_{t_0}+x_n^k$, and vice versa. Thus the result follows semi-continuity of the Milnor number (See Theorem 2.6 in~\cite{greuel2007introduction}). 
\end{proof}

\begin{corollary}
Let $f_t$ be a relative morsification and assume that $n \neq 4$. Then $\mu(f_{t_0}+x_n^k) \geq (k-1)(p-1)^{n-1} + 2\deg(\Delta^\perp(f))$ for every large enough $k$ and for every small enough $t_0$.    
\end{corollary}

\begin{proof}
    Follows directly from Corollary~\ref{cor:mu_rel_mors} and Proposition~\ref{prop:rel_mors_mu_A_1}. 
\end{proof}

\section{Closing Remarks and Questions}\label{section:closing}

\textup{In the case where $\Sing(V(f))$ is a reduced scheme, that is, if $p=2$, we can write $f=\sum_{i,j < n } a_{i,j} x_i x_j$ where $a_{i,j} \in \mathbb{C}\{ \underline{x}\}$ with $a_{i,j}=a_{j,i}$ for every $i,j$. Then, as explained in Example 4.2 in~\cite{kerner2017discriminant}, we have that}
\begin{equation*}
    \Delta ^\perp (f)= \{\det(\{a_{i,j}|_{V(I)}\}_{i,j})=0\}.
\end{equation*}
\textup{Therefore, if the transversal discriminant of $V(f)$ is reduced, we have that $\Delta^\perp(f)$ is exactly the set of $D_\infty$ type points, and the degree of $\Delta^\perp(f)$ (as a Cartier divisor) is exactly the number of such $D_\infty$ points. In addition, if $\Delta^\perp(f) = \emptyset$ then $f$ has only $A_\infty$ points on its singular locus.  Therefore, studying the transversal discriminant generalizes the study of $D_\infty$ points in~\cite{siersma1983isolated,de1988some, pellikaan1990deformations}.} \\

Thus, in the reduced case we have that $Jac_{gen}(f)=I$. So, for $f$ with $\Sing(V(f))=V(I)$ such that its generic transversal type is an ordinary multiple point, we can conclude that:  
\begin{enumerate}
    \item $j(f)=0$ if and only if $\delta(f)=0$ if and only if $f$ is a $A_\infty$ singularity
    \item $j(f)=1$ if and only if $\delta(f)=1$ if and only if $f$ is a $D_\infty$ singularity. 
\end{enumerate}

This reproves Lemma 1.10 and Remark 2.16 in~\cite{pellikaan1990deformations}.  \\

In addition, if we look back at the definition of a relative morsification as in Definition~\ref{def:morsification}, then the remark above, in addition to Theorem~\ref{prop:morsif}, reproves the existence of a relative morsification in the case $p=2$, as presented in~\cite{siersma1983isolated, pellikaan1990deformations}. Therefore, if we reconsider  the proof of Theorem~\ref{thm:main_Jac_gen} (based upon Theorem~\ref{thm:main}) and the proof of Theorem~\ref{thm:delta_trans_bound}, we can conclude that we in fact have equality in the reduced case ($p=2$), that is, $j(f) =\# A_1(f_t) + \deg(\Delta^\perp(f))$ and $\delta(f) = \deg(\Delta^\perp(f))$. This reproves Proposition 2.19 in~\cite{pellikaan1990deformations} and Theorem 4.2 in~\cite{pellikaan1989series}, respectively. \\

Yet, as we have seen in Example~\ref{example:inequality} and in Example~\ref{example:beansbeans}, in general we do not have equalities in Theorem~\ref{thm:main_Jac_gen} and in Theorem~\ref{thm:delta_trans_bound}. One can note that in this case we only have an inequality since the transversal discriminant only measures the topological equaisingularity of the transversal sections of $V(f)$ (as seen in Remark~\ref{rem:top_trans_disc}), but not the analytical equisingularity. Therefore, it would be interesting to find an alternative definition for the transversal discriminant than the one presented in~\cite{kazarian2013discriminant, kerner2017discriminant}, which would be empty if and only if the transversal sections of $V(f)$ along $V(I)$ are analytically equisingular.

\bibliographystyle{plain}
\bibliography{bib}

\end{document}